\documentclass[a4paper,11pt,reqno]{amsart}

\usepackage[T1]{fontenc}
\usepackage{microtype}

\usepackage{amsmath,amssymb,mathtools}
\usepackage{mathrsfs}
\usepackage{bm}
\usepackage{esint}

\usepackage{graphicx}
\usepackage{booktabs}

\usepackage{enumitem}

\usepackage[dvipsnames]{xcolor}
\definecolor{refblue}{RGB}{0,0,255}

\usepackage[
  colorlinks=true,
  linkcolor=refblue,
  citecolor=refblue,
  urlcolor=refblue,
  pdfborder={0 0 0}
]{hyperref}

\usepackage[nameinlink,noabbrev]{cleveref}

\numberwithin{equation}{section}

\theoremstyle{plain}
\newtheorem{theorem}{Theorem}[section]
\newtheorem{lemma}[theorem]{Lemma}
\newtheorem{proposition}[theorem]{Proposition}
\newtheorem{corollary}[theorem]{Corollary}

\newtheorem{remark}[theorem]{Remark}

\theoremstyle{definition}
\newtheorem{definition}[theorem]{Definition}

\crefname{theorem}{Theorem}{Theorems}
\crefname{lemma}{Lemma}{Lemmas}
\crefname{proposition}{Proposition}{Propositions}
\crefname{corollary}{Corollary}{Corollaries}
\crefname{claim}{Claim}{Claims}
\crefname{definition}{Definition}{Definitions}
\crefname{assumption}{Assumption}{Assumptions}
\crefname{example}{Example}{Examples}
\crefname{remark}{Remark}{Remarks}
\crefname{equation}{equation}{equations}
\crefname{section}{Section}{Sections}

\DeclareMathOperator{\supp}{supp}

\usepackage{needspace}

\DeclareUnicodeCharacter{2010}{-}
\DeclareUnicodeCharacter{02BC}{'}
\hypersetup{pdftitle={The Logarithmic Laplacian on Complete Riemannian Manifolds},
  pdfauthor={Rui Chen}}
\title[Logarithmic Laplacian on Manifolds]{The Logarithmic Laplacian on Complete Riemannian Manifolds}
\author{Rui Chen}
\address{School of Mathematical Sciences, Fudan University, Shanghai 200433, China.}
\email{chenrui23@m.fudan.edu.cn}
\begin{document}

\begin{abstract}
We introduce, for the first time, a Bochner representation of the
logarithmic Laplacian on general complete Riemannian manifolds via
spectral calculus and the heat semigroup. More
precisely, letting $P_0=E(\{0\})$ denote the orthogonal projection onto
$\ker(-\Delta)$ and $Q=I-P_0$, and defining the logarithmic operator to
vanish on $\ker(-\Delta)$, we prove
\[
\log(-\Delta)f
=
\int_0^\infty
\frac{e^{-t}Qf-e^{t\Delta}Qf}{t}\,dt,
\]
where the integral converges in the strong improper sense in $L^2$.
This representation recovers the classical pointwise logarithmic
Laplacian on $\mathbb R^n$.

Under a lower Ricci curvature bound and an explicit large-time
heat-kernel integrability condition, we derive a pointwise integral
representation on complete noncompact manifolds. We also compare the
spectral and heat-kernel definitions and identify their difference
explicitly in terms of the loss of heat mass. On $\mathbb H^n$, $n\ge2$, we compute the fractional and logarithmic kernels, determine their asymptotics at zero and infinity, and derive absolute convergence, continuity, and $L^p$ integrability under Dini and weighted integrability assumptions.
\end{abstract}

	\maketitle
\enlargethispage{2pt}
	
	\setcounter{tocdepth}{2}
	\tableofcontents
	
	\section{Introduction and Main Results}

	In recent years, nonlocal operators have emerged at the crossroads of analysis, geometry, probability, and applied mathematics, driving major advances not only in the study of integro‐differential equations \cite{caffarelli2007extension,ros2014dirichlet,chen2019dirichlet,li2023maximum,servadei2015brezis} and  spectral problems \cite{feulefack2022small,laptev2021spectral,dyda2017eigenvalues}, but also in topics as diverse as conformal geometry \cite{jin2014fractional,ludwig2014anisotropic,chang2011fractional}, Lévy‐process dynamics \cite{garbaczewski2018fractional,applebaum2009levy},  nonlocal variational methods \cite{servadei2015brezis,li2025existence,servadei2015fractional} and so on.
	
		Nonlocal operators arise naturally in probability (as generators of jump processes), in physics (in anomalous diffusion and material science), and in geometry (in fractional Yamabe-type problems and nonlocal curvature flows \cite{chambolle2015nonlocal}).  Two of the most prominent examples are the fractional Laplacian and the logarithmic Laplacian.
	
	The fractional Laplacian \((-\Delta)^s\), often viewed as the Riesz fractional derivative \cite{stein1970singular}, has been the subject of intense study over the last decade. Its analytic roots lie in Marcel Riesz’s introduction of the Riesz potential in the late 1940s \cite{riesz1949integrale}, and it was unified with probability theory through Bochner’s subordination principle in the 1950s \cite{bochner1955harmonic,balakrishnan1960fractional}. On \(\mathbb{R}^n\), these viewpoints coincide with the classical Fourier‐multiplier, singular‐integral, and Bochner‐integral definitions of \((-\Delta)^s\) \cite{kwasnicki2017ten}:

	\begin{equation}\label{fours}
		\widehat{(-\Delta)^s u}(\xi)
		=|\xi|^{2s}\,\widehat u(\xi),
	\end{equation}
	\begin{equation}\label{scka}
		(-\Delta)^s u(x)
		=c_{n,s}\:\mathrm{P.V.}\int_{\mathbb{R}^n}\frac{u(x)-u(y)}{|x-y|^{n+2s}}\,dy,
	\end{equation}
	\[
	(-\Delta)^s u
	=\frac{s}{\Gamma(1-s)}\int_0^\infty\bigl(u-e^{t\Delta}u\bigr)\,t^{-1-s}\,dt.
	\]
	It enjoys a maximum principle \cite{li2023maximum,cheng2017maximum}, sharp regularity estimates \cite{ros2014dirichlet,audrito2023neumann,caffarelli2016fractional,stinga2024regularity}, and a spectral theory \cite{feulefack2022small,dyda2017eigenvalues,temgoua2024eigenvalue}.  It has applications 
	in obstacle problems, jump‐process models in finance, and the study of fractional geometry \cite{kreuml2021anisotropic,del2022fractional}.
	
	More recently, the logarithmic Laplacian 
	\[
	\log(-\Delta)\,u
	:=\frac{d}{ds}\Bigl[(-\Delta)^s u\Bigr]_{s=0}
	\]
	was studied in \cite{chen2019dirichlet}, where the authors computed the pointwise representation
	\begingroup\small\begin{equation}\label{duishurn}
		\log \left(-\Delta\right)f(x)  =c_{n}\int_{B_{1}(x)}\frac{f(x)-f(y)}{|x-y|^{n}}\,dy
		-c_{n}\int_{\mathbb{R}^{n}\setminus B_{1}(x)}
		\frac{f(y)}{|x-y|^{n}}\,dy
		+\rho_{n}\,f(x), \quad x\in\mathbb{R}^{n}
	\end{equation}\endgroup
	where
	\begin{equation}\label{cnrn}
		c_{n}:=\pi^{-n/2}\Gamma\!\bigl(\tfrac{n}{2}\bigr)
		=\frac{2}{|\mathbb{S}^{n-1}|}, 
		\qquad
		\rho_{n}:=2\ln 2+\psi\!\bigl(\tfrac{n}{2}\bigr)-\gamma,           
	\end{equation}
	\(\gamma=-\Gamma'(1)\) is the Euler–Mascheroni constant, and
	\(\psi=\Gamma'/\Gamma\) is the digamma function. 
	
	Unlike the fractional Laplacian, whose Fourier symbol \(\lvert\xi\rvert^{2s}\) is everywhere non‐negative, the logarithmic Laplacian is given by the multiplier
	\[
	\widehat{\log(-\Delta)\,u}(\xi)
	=\log\bigl(\lvert\xi\rvert^2\bigr)\,\widehat u(\xi),
	\]
	which fails to be positive‐definite.  This sign‐changing symbol destroys the usual coercivity and complicates the analysis.  Recent work has therefore concentrated on the study of PDEs \cite{chen2023extension,hernandez2025optimal,chen2024cauchy,chen2024positive,zhang2024hopf}, as well as detailed analysis of spectral properties \cite{laptev2021spectral,chen2023bounds}.
	
	A natural question is how to extend these definitions to a general Riemannian manifold \((M,g)\).  Such an extension would allow us to investigate how curvature and topology influence nonlocal operators and to apply logarithmic Laplacian techniques to problems in geometric analysis and mathematical physics.  On a compact manifold, such as a closed manifold, the Laplace–Beltrami operator has discrete spectrum, so by functional calculus, one may immediately define
	\[
	(-\Delta)^s f
	=\sum_{j=0}^\infty \lambda_j^s \langle f,\varphi_j\rangle\,\varphi_j,
	\qquad
	\log(-\Delta)f
	=\sum_{j=1}^\infty \log(\lambda_j)\,\langle f,\varphi_j\rangle\,\varphi_j,
	\]
	where \(\{\varphi_j\}_{j\ge 0}\) is an \(L^2\)-orthonormal family of eigenfunctions with \(-\Delta\varphi_j=\lambda_j\varphi_j\).  
	
	On a noncompact manifold, the spectrum may contain both continuous spectrum and discrete eigenvalues. A discrete eigenfunction expansion is therefore not available in general, and we use the spectral theorem to define \((-\Delta)^s\) and \(\log(-\Delta)\).  
	Since on a general Riemannian manifold there is no global Fourier transform, one cannot use \eqref{fours} to define \((-\Delta)^s\).  Likewise, simply replacing \(|x-y|\) by the geodesic distance \(d(x,y)\) in \eqref{scka} produces only a formal kernel that fails to capture the true fractional Laplacian on \((M,g)\).
	Currently, there are two principal approaches to defining the fractional Laplacian on noncompact manifolds:

	\begin{enumerate}
		\item \textbf{Caffarelli–Silvestre extension.}  One realizes \((-\Delta)^s\) as the Dirichlet–to–Neumann map for a degenerate elliptic problem on the product manifold \(M\times(0,\infty)\), thereby reducing nonlocal questions to local ones in one higher dimension \cite{banica2015some}.
		
		\item \textbf{Bochner integral representation.}  Starting from the abstract spectral measure identity
		\[
		(-\Delta)^s
		=\int_{[0,\infty)} \lambda^s\,dE(\lambda),
		\]
		one rewrites this via the Bochner integral as follows (see \cite{caselli2024asymptotics,caselli2024fractional}):
		\[
		(-\Delta)^s f
		=\frac{s}{\Gamma(1-s)}\int_0^\infty\bigl(f-e^{t\Delta}f\bigr)\,t^{-1-s}\,dt.
		\]
	\end{enumerate}
	
	Research on the fractional Laplacian on Riemannian manifolds is still in its early stages.  Initial efforts extended the Riesz transform to the manifold setting \cite{auscher2004riesz,sikora2004riesz,bakry2014analysis,strichartz1983analysis,auscher2005riesz,carron2007riesz}, and more recently a number of works have begun to develop a systematic theory of fractional Laplacians.  These contributions cover both compact manifolds \cite{d2016fractional,feizmohammadi2024fractional,li2026inverse,de2019fractional,caselli2024fractional} and noncompact manifolds \cite{banica2015some,caselli2024asymptotics,caselli2024fractional,bhowmik2022extension,papageorgiou2024asymptotic,alonso2018integral,papageorgiou2024large,kim2022harnack,anker2026schr}.

To date, no prior work has defined or analyzed the logarithmic Laplacian on a general Riemannian manifold. One conceivable approach would be to adapt the Caffarelli–Silvestre extension method, but even in \(\mathbb{R}^n\) the extension problem for \(\log(-\Delta)\) is substantially more intricate than for \((-\Delta)^s\) \cite{chen2023extension}; carrying it over to a curved background poses formidable technical challenges.  
	
	A more natural strategy is to fall back on the functional calculus definition:
	\[
	\log(-\Delta)
	=\int_{(0,\infty)}\log(\lambda)\,dE(\lambda).
	\]
	The critical obstacle is to connect this abstract spectral measure to the heat kernel via an effective “logarithmic Bochner’’ integral formula.  Without this representation, one cannot derive concrete pointwise kernels or exploit heat‐kernel asymptotics for fine analysis.  
	
	In this paper, we provide a unified definition of the logarithmic Laplacian \(\log(-\Delta)\) on both compact and noncompact complete Riemannian manifolds \((M,g)\) by functional calculus.  On a compact manifold, such as a closed manifold, one may immediately set
	\[
	\log(-\Delta)f
	=\sum_{j=1}^\infty\log(\lambda_j)\,\langle f,\varphi_j\rangle\,\varphi_j
	\]
	using the discrete eigenbasis \(\{\varphi_j\}_{j\ge 0}\) of \(-\Delta\).

Throughout the manifold setting, we assume that $M$ is connected and
complete, without boundary unless boundary conditions are specified.
Let $E$ denote the spectral resolution of the nonnegative self-adjoint
operator $-\Delta$, and put $P_0=E(\{0\}),\,Q=I-P_0.$
Here $P_0$ is the orthogonal projection onto $\ker(-\Delta)$, and $Q$
is the orthogonal projection onto its orthogonal complement. We use the notation $\log(-\Delta)$ for the spectral multiplier
\[
\ell(\lambda)=
\begin{cases}
	\log\lambda, & \lambda>0,\\
	0, & \lambda=0.
\end{cases}
\]
Thus we extend the spectral logarithm by zero on $\ker(-\Delta)$.
This convention agrees with the spectral series over positive
eigenvalues on closed manifolds and with the first-order expansion
of positive fractional powers about $Q$. On its spectral domain in
\[
QL^2(M):=\{Qf:f\in L^2(M)\}=\bigl(\ker(-\Delta)\bigr)^\perp,
\]
the integral representation of $\log(-\Delta)$ is based on the
classical Frullani identity
\[
\log\lambda
=\int_0^\infty\frac{e^{-t}-e^{-\lambda t}}{t}\,dt,
\qquad \lambda>0;
\]
see \cite{bravo2017integrals,gradshteyn2014table}. Throughout this paper, time-integral representations of spectral
operators on their stated domains are understood as strong improper
Bochner integrals in $L^2(M)$. More precisely, we first integrate
over $[a,b]$, where $0<a<b<\infty$, and then let $a$ decrease to zero
and $b$ increase to infinity independently, with convergence in the
$L^2(M)$ norm. We write $\int_0^\infty$ without repeating these endpoint
limits. Absolute Bochner integrability is asserted only when the
$L^2(M)$ norm of the integrand is also known to be integrable in time.

	\begin{theorem}[Bochner formula for $\log(-\Delta)$]\label{logboch11}
For every $f\in H^{\log}(M)$,
\[
\log(-\Delta)f
=\int_{(0,\infty)}\log\lambda\,dE(\lambda)f
=\int_0^\infty\frac{e^{-t}Qf-e^{t\Delta}Qf}{t}\,dt.
\]
Equivalently,
\[
\log(-\Delta)f
=\int_0^\infty
\frac{e^{-t}f-e^{t\Delta}f+(1-e^{-t})P_0f}{t}\,dt.
\]
Here $P_0=E(\{0\})$, $Q=I-P_0$, and $H^{\log}(M)$ is the space in
Proposition~\ref{logarith}.
\end{theorem}

On a complete connected manifold without boundary, every $L^2$ harmonic
function is constant; see \cite[Theorem 2.3]{caselli2024asymptotics}.
Consequently, $P_0=0$ when $\operatorname{Vol}(M)=\infty$, whereas for
finite volume $P_0$ is the projection onto constants, given by
\[
P_0f=\left(\frac{1}{\operatorname{Vol}(M)}
\int_M f\,d\mathrm{vol}\right)\mathbf 1.
\]
Thus Theorem~\ref{logboch11} includes functions with nonzero mean on
finite-volume manifolds. On a closed manifold it agrees with the
spectral series over the positive eigenvalues: the constant component
belongs to $H^{\log}(M)$ and is annihilated by $\log(-\Delta)$.
The details are given below in the discussion of $P_0$ and in
Section~\ref{geometric-settings}.

After completing this paper, we became aware that S. Pramanik \cite{pramanik2025anisotropic} independently defined the operator $\log\left(-\Delta+m\mathbb{I}\right)$ for $m>1$ on closed manifolds for different purposes, and we learned of the recent preprint by Fernández and Saldaña \cite{fernandez2025conformal} on the conformal logarithmic Laplacian on the sphere. We note, however, that these operators differ from ours. In particular, our setting includes noncompact manifolds, and we establish the corresponding properties in what follows.
	
	We also noted the related work of Gerontogiannis and Mesland \cite{gerontogiannis2025logarithmic}, who developed a logarithmic Dirichlet Laplacian on Ahlfors-regular spaces and explored its connections to noncommutative geometry, complementing the results presented here.
	
	\begin{theorem}\label{logrn11}
		Let \(f\in C^{\beta}_{c}(\mathbb{R}^{n})\) for some \(\beta>0\). The logarithmic Laplacian defined by Bochner integral in Theorem \ref{logboch11} admits the pointwise representation
		\[	\log \left(-\Delta\right)f(x)  =c_{n}\int_{B_{1}(x)}\frac{f(x)-f(y)}{|x-y|^{n}}\,dy
		-c_{n}\int_{\mathbb{R}^{n}\setminus B_{1}(x)}
		\frac{f(y)}{|x-y|^{n}}\,dy
		+\rho_{n}\,f(x), \quad x\in\mathbb{R}^{n},   \]
		where $c_n,\rho_n$ are defined in (\ref{cnrn}).
	\end{theorem}

	Moreover, the functional calculus framework allows us to define the natural domains of fractional Laplacian and logarithmic Laplacian (see subsection \ref{hanshu} for detail):
	
	\begin{definition}
		For \(s\ge0\), the fractional Sobolev space on \(M\) is
		\[
		H^{2s}(M)
		:=\Bigl\{\,f\in L^2(M)\;\Big|\;
		\int_0^\infty\lambda^{2s}\,dE_{f,f}(\lambda)<\infty\Bigr\},
		\]
		and the logarithmic Sobolev space is
		\[
		H^{\log}(M)
		:=\Bigl\{\,f\in L^2(M)\;\Big|\;
		\int_{(0,\infty)}(\log\lambda)^2\,dE_{f,f}(\lambda)<\infty\Bigr\}.
		\]
		These are precisely the domains of \((-\Delta)^s\) and \(\log(-\Delta)\), respectively.
	\end{definition}
	
	In the model case \(M=\mathbb{R}^n\), these coincide with the classical definitions:
	\[
	H^s(\mathbb{R}^n)
	=\{\,f\in L^2\left(\mathbb{R}^n\right):\,(1+|\xi|^2)^{s/2}\,\widehat f(\xi)\in L^2\left(\mathbb{R}^n\right)\},\]
	and
	\[H^{\log}(\mathbb{R}^n)
	=\{\,f\in L^2\left(\mathbb{R}^n\right):\,\log(|\xi|^2)\widehat f(\xi)\in L^2\left(\mathbb{R}^n\right)\}.
	\]
	
	\medskip
	
	Functional calculus gives the following $L^2$ convergence results.
On $\mathbb R^n$ they recover the corresponding $L^2$ limits related to
\cite[Theorem 1.1]{chen2019dirichlet} and
\cite[Proposition 4.4]{dinezza2012hitchhiker}; the pointwise and other
$L^p$ conclusions in those works require separate arguments.

	\begin{proposition}\label{slx11}
		For every \(f\in H^\varepsilon(M)\), \(\varepsilon>0\),
		\[
		\lim_{s\to0^+}\bigl\|(-\Delta)^s f - f + E(\{0\})f\bigr\|_{L^2(M)} \;=\; 0.
		\]
	\end{proposition}
	
	\begin{proposition}\label{slx22}
		For every \(f\in H^2(M)\),
		\[
		\lim_{s\to1^-}\bigl\|(-\Delta)^s f + \Delta\: f\bigr\|_{L^2(M)} \;=\;0.
		\]
	\end{proposition}
	
	\begin{proposition}\label{slx33}
		If \(f\in H^\varepsilon(M)\,\cap\,H^{\log}(M)\), then
		\[
		\frac{(-\Delta)^s - I + E(\{0\})}{s}\,f
		\xrightarrow{s\to0^+}
		\log(-\Delta)\,f
		\quad\text{in }L^2(M).
		\]
	\end{proposition}

	Naturally, one would like to understand the relationship between the fractional Sobolev spaces and the logarithmic Sobolev space.  The next result gives a sharp dichotomy. 
	
	\begin{theorem}\label{spect11}
		Set
		\[
		\delta \;=\;\inf\bigl(\sigma(-\Delta)\setminus\{0\}\bigr)\ge0.
		\]
		(i) If \(\delta>0\), then, for every \(\varepsilon>0\), one has the continuous embedding
		\[
		H^{2\varepsilon}(M)\;\subset\;H^{\log}(M).
		\]
		(ii) If $\delta=0$, then, for every $\varepsilon>0$,
\[
H^{2\varepsilon}(M)\not\subset H^{\log}(M).
\]
	\end{theorem}

	\begin{remark}
		Theorem \ref{spect11} shows that if 
		\(\displaystyle \delta=\inf(\sigma(-\Delta)\setminus\{0\})=0\),
		then the inclusion 
		\[
		H^{2\varepsilon}(M)\;\not\subset\;H^{\log}(M)
		\]
		holds for any \(\varepsilon>0\).  This does not, however, answer the question of whether the space of compactly supported smooth functions \(C_c^\infty(M)\) is contained in \(H^{\log}(M)\). In fact, the answer depends on the exact form of the spectral measure in each case, and must be determined by a detailed analysis of its behavior near \(\lambda=0\).
	\end{remark}
	
The above results hold for a general underlying space $M$, including Riemannian manifolds and discrete metric spaces. In what follows we assume that $M$ is a general complete Riemannian manifold equipped with the metric $g$. One may define the heat‐kernel fractional Laplacian by
$$
(-\Delta)^s_{\mathrm{hk}}f(x)
=\frac{s}{\Gamma(1-s)}
\int_M\bigl(f(x)-f(y)\bigr)\,\mathcal{K}_s(x,y)\,d\mathrm{vol}(y)
$$ under the mild integrability conditions
	$$
	\mathcal{K}_s(x,y)
	:=\int_0^\infty\frac{p_t(x,y)}{t^{1+s}}\,dt,
	\quad
	\int_0^\infty\!\int_M\frac{\lvert f(x)-f(y)\rvert}{t^{1+s}}\,p_t(x,y)\,d\mathrm{vol}(y)\,dt<\infty.
	$$
	Meanwhile, on any complete $(M,g)$ the spectral formula is
\[
(-\Delta)^s_{\mathrm{spec}}f
=\frac{s}{\Gamma(1-s)}
\int_0^\infty\frac{f-e^{t\Delta}f}{t^{1+s}}\,dt
\quad\text{in }L^2(M),\qquad f\in H^{2s}(M).
\]
Under the stated absolute pointwise integrability conditions this agrees
almost everywhere with the time-integral representative. Pointwise
statements below refer to that representative.
	
	When $M$ is stochastically complete, for example, if $\mathrm{Ric}_g\ge -k,k\ge 0$, these two constructions coincide.  It is natural to ask: on a non-stochastically complete manifold, how do $(-\Delta)^s_{\rm hk}$ and $(-\Delta)^s_{\rm spec}$ differ?  A key benefit of our functional calculus framework is that it not only produces both definitions in a unified setting, but also allows us to quantify their exact discrepancy in terms of the manifold’s mass‐loss function.
	
	A straightforward computation shows that
	\[
	(-\Delta)^s_{\rm spec}f(x)
=(-\Delta)^s_{\rm hk}f(x)+V_s(x)f(x),
	\]
	where
	\[
		V_s(x)
	=\frac{s}{\Gamma(1-s)}
	\int_0^\infty t^{-1-s}\,r(t,x)\,dt,
	\quad r(t,x):=1-\int_M p_t(x,y)\,d\mathrm{vol}(y)
.
	\]
	The limit of the multiplication term is described by the next proposition.
Its relation to the spectral limit is in $L^2_{\mathrm{loc}}$, as explained
in the proof below; $L^2$ convergence alone does not imply convergence at
every point.

	\begin{proposition}\label{chaz11}
		With notation as above,
		\[
		\lim_{s\to0^+}V_s(x)=1-H(x),
		\quad
		x\in M,
		\]
		where $$H(x)=\lim_{t\to\infty}\int_M p_t(x,y)\,d\mathrm{vol}(y)$$ is the non‐explosion probability defined in (\ref{taoyi}).
	\end{proposition}
	
	The condition $H\equiv1$ is equivalent to stochastic completeness,
since $1\ge P_t1\ge H$. Likewise $V_s\equiv0$ for any fixed $s\in(0,1)$
is equivalent to preservation of heat mass. Thus the spectral and
heat-kernel constructions coincide on their common pointwise domain
when $M$ is stochastically complete.
For comparison, \cite[Theorem 1.2]{caselli2024asymptotics} gives an
almost-everywhere limit for the singular-integral definition on
infinite-volume manifolds, under its stated Sobolev and bounded-support
assumptions. Theorem~1.3 of that work proves equality of the definitions
and $L^2$ limits in the stochastically complete case; Proposition~1.4
characterizes stochastic completeness through the limiting heat mass in
infinite volume. Its fractional parameter is $s/2$.

Now we derive a pointwise formula. The operator
$\log(-\Delta):H^{\log}(M)\to L^2(M)$ is bounded with respect to the
graph norm, but is in general unbounded on $L^2(M)$. A lower Ricci
curvature bound guarantees stochastic completeness, but does not by
itself control the large-time part of the logarithmic integral. We
therefore impose the following heat-tail condition: for every nonempty
compact set $K\subset M$, let
\begin{equation}\label{heat-tail}
q_K(t):=\sup_{x\in K}p_t(x,x),\qquad
\int_1^\infty \sqrt{q_K(t)}\,\frac{dt}{t}<\infty.
\tag{HT}
\end{equation}

\begin{theorem}\label{thm:log-noncompact11}
Let $(M,g)$ be a complete connected Riemannian manifold with
$\mathrm{Ric}_g\ge-(n-1)kg$, $k\ge0$, satisfying \eqref{heat-tail}.
If $f\in C_c^\alpha(M)$ for some $\alpha>0$, then $f\in H^{\log}(M)$,
and the spectral logarithmic Laplacian has the continuous representative
\begin{align*}
\log(-\Delta)_{\mathrm{spec}}f(x)
&=
\int_M K_1(x,y)[f(x)-f(y)]\,d\mathrm{vol}(y)
-\int_M K_2(x,y)f(y)\,d\mathrm{vol}(y)
+\Gamma'(1)f(x),
\end{align*}
where
\[
K_1(x,y)=\int_0^1p_t(x,y)\,\frac{dt}{t},
\qquad
K_2(x,y)=\int_1^\infty p_t(x,y)\,\frac{dt}{t}.
\]
Both integrals involving $f$ are absolutely convergent for every
$x\in M$.
\end{theorem}

Condition~\eqref{heat-tail} is a large-time decay assumption and is
strictly weaker than the existence of a positive spectral gap. Indeed,
if $\lambda_0:=\inf\sigma(-\Delta)>0,$
then the semigroup property and spectral calculus give
\[
p_t(x,x)
\le
e^{-\lambda_0(t-1)}p_1(x,x),
\qquad t\ge1,
\]
and hence \eqref{heat-tail} holds. More generally, any estimate
\[
q_K(t)\le C_K t^{-\theta},
\qquad t\ge1,
\]
with $\theta>0$, is sufficient. Thus \eqref{heat-tail} holds on $\mathbb R^n$, where $p_t(x,x)=(4\pi t)^{-n/2},$
and on $\mathbb H^n$, $n\ge2$, where the bottom of the spectrum is
positive. It also holds on products $\mathbb R^m\times N,$
with $m\ge1$ and $N$ compact, since the diagonal heat kernel has
large-time decay of order $t^{-m/2}$. More generally, it is satisfied
whenever a large-time on-diagonal estimate
\[
p_t(x,x)\le \frac{C_K}{\operatorname{vol}(B(x,\sqrt t))}
\]
holds for $x\in K$ and the volume has a positive lower
growth $\operatorname{vol}(B(x,\sqrt t))\ge c_K t^\theta$
for some $\theta>0$.

In contrast, \eqref{heat-tail} fails on every complete connected
finite-volume manifold. Indeed, stochastic completeness and
Cauchy--Schwarz give
\[
p_{2t}(x,x)
=
\int_M p_t(x,y)^2\,d\mathrm{vol}(y)
\ge
\frac{1}{\operatorname{vol}(M)}
\left(\int_Mp_t(x,y)\,d\mathrm{vol}(y)\right)^2
=
\frac{1}{\operatorname{vol}(M)}.
\]
Consequently, $\int_1^\infty\sqrt{q_K(t)}\,\frac{dt}{t}=\infty.$
This includes all compact manifolds and, for example, finite-volume
hyperbolic manifolds. In these cases the obstruction is precisely the
zero mode. The logarithmic operator is instead treated by projecting
away from $\ker(-\Delta)$, as in Theorem~\ref{logboch11}.

\smallskip

	On a general manifold \((M,g)\) without a curvature condition, we can similarly compare the spectral logarithmic Laplacian and its heat‐kernel counterpart. However, in the absence of curvature bounds, one lacks the needed heat‐kernel estimates; we thus require that
	\[
	K_1(x,y)\;=\;\int_0^1\frac{p_t(x,y)}{t}\,dt,
	\qquad
	K_2(x,y)\;=\;\int_1^\infty\frac{p_t(x,y)}{t}\,dt
	\]
	be well-defined (i.e.\ each integral converges for all \(x\neq y\)), and that $f$ satisfies
	\[
	\int_0^1\int_M\bigl|f(x)-f(y)\bigr|\frac{p_t(x,y)}{t}\,d\mathrm{vol}(y)\,dt
	\;<\;\infty,
	\quad
	\int_1^\infty\int_M|f(y)|\frac{p_t(x,y)}{t}\,d\mathrm{vol}(y)\,dt<\infty.
	\]
	Under these absolute integrability hypotheses, one may interchange integrals and define the heat-kernel logarithmic Laplacian by
	\begingroup\small
	\[
	\bigl(\log(-\Delta)\bigr)_{\mathrm{hk}}f(x)
	:= \int_M K_1(x,y)\,\bigl(f(x)-f(y)\bigr)\,d\mathrm{vol}(y)
	\;-\;\int_M K_2(x,y)\,f(y)\,d\mathrm{vol}(y)
	\;+\;\Gamma'(1)\,f(x).
	\]
	\endgroup
	If, in addition, $f\in H^{\log}(M)$ and $P_0f=0$, the truncated
spectral formula identifies this expression almost everywhere, and gives
the decomposition below. Its pointwise interpretation is the time-integral
representative wherever the absolute integrability conditions hold.
	\[
	\bigl(\log(-\Delta)\bigr)_{\mathrm{spec}}f(x)
=\bigl(\log(-\Delta)\bigr)_{\mathrm{hk}}f(x)+V(x)f(x),
	\]
	where
	\[
	V(x)\;=\;\int_0^1t^{-1}\,r(t,x)\,dt,
	\quad
	r(t,x)=1-\int_Mp_t(x,y)\,d\mathrm{vol}(y).
	\]

	For each $x$, $V(x)$ is finite: a smooth compactly supported cutoff
$0\le\chi\le1$ with $\chi(x)=1$ gives
$0\le r(t,x)\le1-P_t\chi(x)\le t\|\Delta\chi\|_\infty$.
If $M$ is stochastically complete, $V=0$, and the two logarithmic
expressions agree under the preceding domain and convergence hypotheses.

	Finally, we specialize to the prototypical noncompact setting of real hyperbolic space \(\mathbb{H}^n,\:n\ge 2\), which is stochastically complete with $\mathrm{Ric}_g=-\left(n-1\right)$. Thanks to the explicit form and sharp asymptotics of the heat kernel on \(\mathbb{H}^n\), we are able to derive precise pointwise formulas and kernel bounds for both fractional and logarithmic Laplacians.
	
	\begin{proposition}\label{fensjx11}
		The fractional Laplacian kernel \(\mathcal{K}_s(r)\) on \(\mathbb{H}^n\) defined in (\ref{fenshureh}) satisfies
		\[
		\mathcal{K}_{s}(r)\sim r^{-n-2s},\quad r\to0;
		\qquad
		\mathcal{K}_{s}(r)\sim r^{-1-s}e^{-(n-1)r},\quad \text{as}\quad r\to\infty.
		\]
	\end{proposition}
	
	These asymptotic behaviors agree with the kernel estimates of the fractional Laplacian obtained via the Fourier transform on \(\mathbb{H}^n\) in \cite[Theorem 2.4]{banica2015some}.
	
	\begin{proposition}\label{prop:K1K211}
		The logarithmic kernels on \(\mathbb{H}^n\) introduced in \eqref{loghe} satisfy
		\[
		K_1(r)\sim r^{-n},\;r\le1;
		\quad
		K_1(r)\sim r^{\frac{n-5}{2}}e^{-\frac{(n-1)r}{2}}e^{-\frac{r^2}{4}},\;r\ge1,
		\]
		and
		\[
		K_2(r)\sim 1,\;r\le1;
		\quad
		K_2(r)\sim r^{-1}e^{-(n-1)r},\;r\ge1.
		\]
	\end{proposition}
	
	With these estimates in hand, we can weaken the hypotheses needed for a pointwise definition of \(\log(-\Delta_{\mathbb{H}^n})f\) in Theorem \ref{thm:log-noncompact11}, see Proposition \ref{prop:cont-logDelta11}.
	
	\begin{proposition}\label{prop:cont-logDelta11}
		Let $f\in L^1_w(\mathbb H^n)$ be locally Dini continuous at a
point $x$. Then the formula
		\begingroup\small\[
		\log(-\Delta_{\mathbb{H}^n})f(x)
		=\int_{\mathbb{H}^n}K_1\bigl(d(x,y)\bigr)\bigl[f(x)-f(y)\bigr]\,d\mathrm{vol}(y)
		-\int_{\mathbb{H}^n}K_2\bigl(d(x,y)\bigr)\,f(y)\,d\mathrm{vol}(y)
		+\Gamma'(1)\,f(x)
		\]
		\endgroup
		converges absolutely at that point. If, in addition, $f$ is
locally uniformly Dini continuous on $\mathbb H^n$, the formula defines
a continuous function of $x$. Here local uniform Dini continuity means
that every point has a neighborhood $\Omega$ with
$\int_0^1\omega_{f,\Omega}(r)\,dr/r<\infty$.
	\end{proposition}
	
For functions outside the spectral $L^2$ domain this formula is a
pointwise extension, not an assertion of spectral-domain membership.
The definitions of $L^1_w(\mathbb H^n)$ and the Dini conditions are in
Subsection~\ref{last}.	Thus, the combination of explicit hyperbolic heat‐kernel asymptotics and Dini‐type continuity yields a robust pointwise theory of the logarithmic Laplacian on \(\mathbb{H}^n\).

	In the final part of the paper, we exploit the explicit hyperbolic‐space estimates to obtain a refined splitting of the logarithmic Laplacian into three well‐understood pieces, directly paralleling the Euclidean formula \eqref{duishurn}.  Fix \(x\in\mathbb{H}^n\) and write \(r=d(x,y)\), \(B_1(x)=\{y:r<1\}\).  We define the remainder
	\begingroup\small	
	\begin{equation}\label{remind-intro}
		\begin{aligned}
			\mathcal R_{n}(f;x)
			:=&-\int_{B_{1}(x)}K_{2}(r)\,f(y)\,
			d\mathrm{vol}(y)
			-\int_{\mathbb H^{n}\setminus B_{1}(x)}
			K_{1}(r)f(y)d\mathrm{vol}(y)\\+&\left(\int_{\mathbb H^{n}\setminus B_{1}(x)}	K_{1}(r)\,d\mathrm{vol}(y)+\Gamma'(1)\right)f(x).
		\end{aligned}
	\end{equation}
	\endgroup
	Using the kernel bounds of Proposition~\ref{prop:K1K211} and
the uniform row and column integral bounds proved below, one shows that \(\mathcal R_n(f;\cdot)\in L^p(\mathbb{H}^n)\) whenever \(f\in L^p(\mathbb{H}^n)\), \(1\le p\le\infty\). 
	
	Therefore, we may express the logarithmic Laplacian in the following pointwise form.
	
	\begin{proposition}\label{prop:split-refined11}
		Let $f\in L_{w}^1(\mathbb H^n)$ and suppose that \(f\) is locally Dini continuous at the point $x\in \mathbb H^{n}$. 
		Then,
		\begin{equation}\label{eq:split-refined}
			\begin{aligned}
				(\log(-\Delta_{\mathbb H^{n}})f)(x)
				&=\int_{B_{1}(x)}K_{1}(r)\bigl[f(x)-f(y)\bigr]\,d\mathrm{vol}(y)\\
				&\quad
				-\int_{\mathbb H^{n}\setminus B_{1}(x)}K_{2}(r)\,f(y)\,
				d\mathrm{vol}(y)
				+\mathcal R_{n}(f;x).
			\end{aligned}
		\end{equation}
	\end{proposition}

	Finally, by a careful use of Dini‐continuity near the diagonal and the exponential decay at infinity,  the full pointwise formula defines an \(L^p\)-function for every \(1<p\le\infty\).

	\begin{proposition}\label{logint11}
		Let \(f\in C_c(\mathbb H^n)\) and suppose that $f$ is uniformly Dini continuous on $\mathbb{H}^n.$ Then $$\log(-\Delta_{\mathbb H^n})f\in L^p(\mathbb H^n),\:1<p\le \infty.$$
	\end{proposition}

	The paper is organized as follows.  In Section 2, we introduce the functional calculus definitions of \((-\Delta)^s\) and \(\log(-\Delta)\), define the corresponding fractional and logarithmic Sobolev spaces, and analyze their interrelation.  We also establish the Bochner integral formula for the logarithmic Laplacian and prove that it recovers the classical pointwise representation in \(\mathbb{R}^n\). In Section 3, we first extend both fractional and logarithmic Laplacians to compact and noncompact settings, then compare the spectral and heat‐kernel definitions for \((-\Delta)^s\) and \(\log(-\Delta)\).  Finally, under a Ricci lower bound and \eqref{heat-tail}, we derive
the pointwise integral formula for the logarithmic Laplacian.
	Finally, Section 4 specializes to \(\mathbb{H}^n\), \(n\ge2\), where explicit heat‐kernel asymptotics yield precise estimates for both fractional and logarithmic kernels.  We then give weighted Dini conditions for its pointwise definition and prove continuity and \(L^p\) integrability on hyperbolic space.

	Throughout this paper, we write \(f \sim g,\:s_0<s<s_1\) with $s_0,s_1\in [0,\infty]$ to mean that there exists a constant \(C>0\) such that
	\[
	C^{-1}\,f(s)\;\le\;g(s)\;\le\;C\,f(s)
	\quad\text{holds for all }s\in \left(s_0,s_1\right).
	\]
	We also write \(f \lesssim g\),  (resp. \(f \gtrsim g\)), $s_0<s<s_1$ to mean that there exists a constant \(C>0\)
	such that \(f(s) \le C\,g(s)\) (resp. \(f(s) \ge C\,g(s)\)) for all \(s_0<s<s_1.\)

	\section{Functional Calculus and Integral Formulas}
	
	In this section we define fractional and logarithmic Laplacians by
functional calculus. The heat-semigroup representation of fractional
powers is classical; see \cite{stinga2010extension,roncal2016fractional,stinga2019user}.
For the logarithmic multiplier we apply the scalar Frullani identity and
keep the projection $Q$ explicit. The resulting time integral is a strong
improper integral on $H^{\log}(M)$ and defines a bounded map from its
graph norm into $L^2(M)$. We also study the basic properties of the fractional Sobolev spaces $H^s\left(M\right)$ and their logarithmic analogues $H^{\log}\left(M\right)$ arising from this calculus, and we analyze the relationship between the fractional and logarithmic Laplacian from the perspective of spectral calculus. 
	
	Finally, by applying our spectral formula in the Euclidean setting we rederive the classical expression for the logarithmic Laplacian on $\mathbb R^n$.  Remarkably, the expression for the logarithmic Laplacian derived here using the Bochner integral formula is in complete agreement with the form obtained by H.~Chen and T.~Weth in \cite{chen2019dirichlet}.  Although one could have anticipated this agreement, the path via the Bochner integral and spectral calculus is both conceptually clean and technically subtle.

	Note that the functional calculus framework we develop below applies to any nonnegative self-adjoint operator on $L^2(M,\mu)$ with spectrum in $[0,\infty)$.  In particular, one may take $M$ (with its natural measure $\mu$) to be:
	
	\begin{enumerate}
		\item Euclidean space $\mathbb R^N$ (the usual Laplacian on $L^2(\mathbb R^N)$);
		\item a bounded domain $\Omega\subset\mathbb R^N$ (Dirichlet or Neumann Laplacian on $L^2(\Omega)$);
		\item a closed (compact, boundaryless) Riemannian manifold $M$ (Laplace–Beltrami on $L^2(M)$);
		\item a compact Riemannian manifold with boundary (Dirichlet/Neumann Laplacian on $L^2(M)$);
		\item a noncompact complete Riemannian manifold $M$ (Laplace–Beltrami on $L^2(M)$).
	\end{enumerate}
	
	In all cases we denote the operator simply by
	$$
	-\Delta \colon \mathrm{Dom}(-\Delta)\subset L^2(M,\mu)\;\longrightarrow\;L^2(M,\mu).
	$$
	
	\subsection{Framework of functional calculus}\label{hanshu}
	
	By the spectral theorem there is a unique projection-valued measure
	\(
	E
	\)
	supported on \([0,\infty)\) such that \cite{reed1981functional}
	\[
	-\Delta=\int_{[0,\infty)}\lambda\,\mathrm dE(\lambda),
	\]
	i.e.\ for any $f\in\operatorname{Dom}(-\Delta)$ and $g\in L^{2}(M,\mu)$,
	\[
	\langle -\Delta f,g\rangle_{L^{2}(M,\mu)}
	=\int_{[0,\infty)}\lambda\,\mathrm dE_{f,g}(\lambda).
	\]
	For each pair \(f,g\in L^2(M,\mu)\), $E_{f,g}$ is a regular complex Borel measure of bounded variation on \([0,\infty)\), supported on \(\sigma(-\Delta)\),  where
	\[
	E_{f,g}(B)\;=\;\bigl\langle E(B)\,f,\;g\bigr\rangle_{L^2(M,\mu)},\:\:E\left(\sigma\left(-\Delta\right)\right)=I \quad\text{where}\:\:
	B\subset[0,\infty)\text{ is a Borel set}
	\]
	and satisfies
	\[
	\bigl|E_{f,g}\bigr|\bigl([0,\infty)\bigr)
	\;\le\;\|f\|_{L^2(M,\mu)}\,\|g\|_{L^2(M,\mu)}, 
	\]
	Moreover, $E_{f,f}$ is a positive measure with $E_{f,f}\bigl([0,\infty)\bigr)=\|f\|_{L^2(M,\mu)}^2.$
	
	For each measurable function $\varphi:[0,\infty)\to\mathbb{C}$, we set
	\begin{equation}\label{spec}
		\varphi(-\Delta)=\int_{[0,\infty)}\varphi(\lambda)\,\mathrm dE(\lambda)
	\end{equation}
	with  
	
	\[
	\operatorname{Dom}\bigl(\varphi(-\Delta)\bigr)
	=\Bigl\{f\in L^{2}(M,\mu):
	\int_{[0,\infty)}|\varphi(\lambda)|^{2}\,\mathrm dE_{f,f}(\lambda)<\infty
	\Bigr\}.
	\]
	The action is given by  
	\[
	\bigl\langle\varphi(-\Delta)f,g\bigr\rangle_{L^{2}(M,\mu)}
	=\int_{[0,\infty)}\varphi(\lambda)\,\mathrm dE_{f,g}(\lambda), \quad f\in \operatorname{Dom}\bigl(\varphi(-\Delta)\bigr),g \in L^{2}(M,\mu)
	\]
	and
	\[||\varphi\left(-\Delta\right)f||_{L^{2}(M,\mu)}^2=\int_{[0,\infty)}|\varphi(\lambda)|^{2}\,\mathrm dE_{f,f}(\lambda). \]
	
	In particular, $f\in \operatorname{Dom}\left(-\Delta\right)$ if and only if $\int_{[0,\infty)}\lambda^{2}\,\mathrm dE_{f,f}(\lambda)<\infty$ with
	\[||-\Delta \:f||_{L^{2}(M,\mu)}^2=\int_{[0,\infty)}\lambda^{2}\,\mathrm dE_{f,f}(\lambda).\]
	If $\varphi$ is bounded, then $$\|\varphi(-\Delta)\|=\mathop{\mathrm{ess\,sup}}_{E}|\varphi|,$$
	where the essential supremum is taken with respect to the spectral measure $E$. If $\varphi$ is real–valued (resp.\ non–negative), 
	then $\varphi(-\Delta)$ is self–adjoint (resp.\ positive).
	
	Below, by selecting appropriate measurable functions $\varphi$, we define the fractional Laplacian and logarithmic Laplacian via the functional calculus.

	For $s\ge 0,$ taking $\varphi(\lambda)=\lambda^s$, define
	\[
	(-\Delta)^{s}:
	\;=\;
	\int_{[0,\infty)}\lambda^{s}\,\mathrm dE(\lambda),
	\]
	with 
	\[\operatorname{Dom}\left(\left(-\Delta\right)^s\right)=\Bigl\{f\in L^{2}(M,\mu):
	\int_{[0,\infty)}\lambda^{2s}\,\mathrm dE_{f,f}(\lambda)<\infty
	\Bigr\}:=H^{2s}(M),\]
	Endowed with the inner product
	\[
	\langle f,g\rangle_{H^{2s}\left(M\right)}
	\;=\;
	\int_{[0,\infty)}(1+\lambda)^{2s}\,\mathrm dE_{f,g}(\lambda),
	\qquad
	\|f\|_{H^{2s}\left(M\right)}^2
	=\langle f,f\rangle_{H^{2s}\left(M\right)},\] 
	It is obvious that \(H^{s}(M)\) is an inner product space and
	\[
	f\in H^{2s}(M)
	\;\Longleftrightarrow\;
	(-\Delta)^{s}f\in L^{2}(M,\mu),\: f\in L^{2}(M,\mu)
	\]
	since the homogeneous Sobolev seminorm
	\[
	\|(-\Delta)^{s}f\|_{L^{2}\left(M,\mu\right)}^{2}
	=
	\int_{[0,\infty)}\lambda^{2s}\,\mathrm dE_{f,f}(\lambda).
	\]
	Thus, $H^{s}\left(M\right)$ is continuously embedded in $L^{2}\left(M,\mu\right).$ In particular, if $-\Delta$ is the Laplace--Beltrami operator on a noncompact complete Riemannian manifold, then $$C_c^{\infty}\left(M\right)\subset H^{2}\left(M\right)=\operatorname{Dom}\left(-\Delta\right)\subset H^{2s}\left(M\right),s\in \left(0,1\right).$$

	\begin{proposition}
		\(H^{s}(M), s\ge0\) is a Hilbert space. 
	\end{proposition}
	
	\begin{proof}
		It is equivalent to show \(H^{2s}(M), s\ge0\) is a Hilbert space. Define
		\[
		T\colon H^{2s}(M)\;\longrightarrow\;L^2(M,\mu),
		\qquad
		T(f)\;=\;(I-\Delta)^s f.
		\]
		By the spectral theorem,
		\[
		\|T(f)\|_{L^2}^2
		=\|(I-\Delta)^s f\|_{L^2}^2
		=\int_{[0,\infty)}(1+\lambda)^{2s}\,dE_{f,f}(\lambda)
		=\|f\|_{H^{2s}}^2,
		\]
		so \(T\) is an isometry from \(\bigl(H^{2s}(M),\|\cdot\|_{H^{2s}}\bigr)\) onto its image
		\[
		\mathcal M = T\bigl(H^{2s}(M)\bigr)\subset L^2(M,\mu).
		\]
		
		Let \(\{g_n\}\subset\mathcal M\) be a sequence converging in \(L^2\) to some \(g\).  By definition of \(\mathcal M\), each \(g_n=(1-\Delta)^s f_n\) for some \(f_n\in H^{2s}(M)\).  Since $g_n\to g$ in $L^2,$
		and the operator \((1-\Delta)^{-s}\) is bounded on \(L^2\) (its spectral multiplier \((1+\lambda)^{-s}\) is bounded), the sequence $f_n
		=(1-\Delta)^{-s}g_n$
		converges in \(L^2\) to $	f:=(1-\Delta)^{-s}g.$
		Moreover, \(f\in H^{2s}(M)\) because
		\[
		\|f\|_{H^{2s}}
		=\|(1-\Delta)^s f\|_{L^2}
		=\|g\|_{L^2}
		<\infty.
		\]
		Finally, \(T(f)=(1-\Delta)^s f = g\).  Hence \(g\in\mathcal M\), proving \(\mathcal M\) is closed. It follows that \(H^{2s}(M)\) is a Hilbert space.
	\end{proof}
	
	Suppose the bottom of the spectrum of $-\Delta$ is \(\lambda_{0}>0\), then by the functional calculus,
	\[
	\|(-\Delta)^{s}f\|_{L^{2}\left(M,\mu\right)}^{2}
	=\int_{[0,\infty)}\lambda^{2s}\,dE_{f,f}(\lambda)
	\;\ge\;
	\int_{[\lambda_0,\infty)}\lambda^{2s}\,dE_{f,f}(\lambda)
	\ge \lambda_{0}^{2s}\,\|f\|_{L^{2}\left(M,\mu\right)}^{2}.
	\]
	Hence
	\[
	\|(-\Delta)^{s}f\|_{L^{2}\left(M,\mu\right)}
	\;\ge\;
	\lambda_{0}^{\,s}\,\|f\|_{L^{2}\left(M,\mu\right)},
	\]
	for all $f\in H^{2s}\left(M\right)$. Hence, the homogeneous Sobolev seminorm and the \(H^{2s}\)‐norm are equivalent on \(H^{2s}\left(M\right)\). The condition \(\inf\sigma(-\Delta)=\lambda_{0}>0\) is satisfied on some spaces, such as real hyperbolic space $\mathbb{H}^n, n\ge 2$, the Laplace–Beltrami operator $-\Delta_{\mathbb{H}^n}$ on $L^2(\mathbb{H}^n)$ has purely continuous spectrum given by
	\[
	\sigma\bigl(-\Delta_{\mathbb{H}^n}\bigr)
	\;=\;
	\bigl[\frac{\left(n-1\right)^2}{4},\;\infty\bigr).
	\]
	
	In the following, we provide the definitions of the logarithmic Laplacian and the corresponding logarithmic Sobolev space. By the spectral theorem, for each \(s>0\), we have
	\[
	(-\Delta)^{s}f
	=\int_{[0,\infty)}\lambda^{s}\,dE(\lambda)\,f,\quad f\in H^{2s}\left(M\right),
	\]
	so that
	\[
	(-\Delta)^{s}f - f+E\left(\left\{0\right\}\right)f
	=\int_{(0,\infty)}\bigl(\lambda^{s}-1\bigr)\,dE(\lambda)\,f.
	\]
	Since
	\[
	\frac{\lambda^{s}-1}{s}
	\;\xrightarrow{s\to0^+}\;\log \lambda,\lambda>0.
	\]
	It is natural to define the logarithmic Laplacian by
	\[
	\log(-\Delta)
	:=\int_{(0,\infty)}\log\lambda\;dE(\lambda).
	\]
	By spectral calculus,
	\[
	\operatorname{Dom}\bigl(\log(-\Delta)\bigr)=
	\Bigl\{\,f\in L^{2}(M,\mu)\;\Big|\;
	\int_{(0,\infty)}(\log\lambda)^{2}\,dE_{f,f}(\lambda)<\infty
	\Bigr\}.
	\]
	
	\begin{proposition}\label{logarith}
		Equipped with the inner product
		\[
		\langle f,g\rangle_{\log}
		:=\langle f,g\rangle_{L^{2}}
		+\bigl\langle\log(-\Delta)f,\;\log(-\Delta)g\bigr\rangle_{L^{2}},
		\]
		the space
		\[
		\operatorname{Dom}\bigl(\log(-\Delta)\bigr)
		=\Bigl\{f\in L^{2}(M,\mu)\;\Big|\;\int_{(0,\infty)}(\log\lambda)^{2}\,dE_{f,f}(\lambda)<\infty\Bigr\}:=H^{\log}(M)
		\]
		is a Hilbert space.
	\end{proposition}
	
	\begin{proof}
		Define the operator
		\[
		A
		:=\bigl(I+(\log(-\Delta))^{2}\bigr)^{1/2}
		\]
		by functional calculus, so that for \(f\in\operatorname{Dom}(\log(-\Delta))\),
		\[
		A\,f
		=P_0f+\int_{(0,\infty)}\sqrt{1+(\log\lambda)^2}\,dE(\lambda)f,
		\]
		and
		\[
		\|A\,f\|_{L^{2}}^{2}
		=\|P_0f\|_2^2+\int_{(0,\infty)}\bigl(1+(\log\lambda)^2\bigr)\,dE_{f,f}(\lambda)
		=\|f\|_{\log}^{2}.
		\]
		Thus the map
		\[
		T\colon H^{\log}(M)
		\;\longrightarrow\;L^{2}(M,\mu),
		\qquad
		T(f)=A\,f,
		\]
		is an isometry onto its image
		\(\mathcal{N}:=T\bigl(H^{\log}(M)\bigr)\).
		
		Let \(\{g_n\}\subset\mathcal{N}\) converge in \(L^{2}\) to some \(g\).  By definition \(g_n = A f_n\) for some \(f_n\in\operatorname{Dom}(\log(-\Delta))\).  Since \(A^{-1}\) exists and is bounded on \(L^{2}\), the sequence $f_n = A^{-1} g_n$
		converges in \(L^{2}\) to $	f = A^{-1} g.$
		Moreover,
		\(\|f\|_{\log}^2 = \|A f\|_{L^2}^2 = \|g\|_{L^2}^2 < \infty\),
		so \(f\in\operatorname{Dom}(\log(-\Delta))\) and \(T(f)=g\).  Thus \(g\in\mathcal{N}\), proving \(\mathcal{N}\) is closed. Hence $H^{\log}(M)$ is a Hilbert space.
	\end{proof}

	Below we turn to a more detailed analysis of the logarithmic Sobolev space \(H^{\log}(M)\).

\begin{proof}[\textbf{Proof of
	Theorem~\ref{spect11}.}]
		\textbf{(1) The case \(\delta>0\).} Note that
		\[
		\int_{(0,\infty)}(\log\lambda)^2\,dE_{f,f}
		=\int_{[\delta,\infty)}(\log\lambda)^2\,dE_{f,f}.
		\]
		On $[\delta,1]$, we have $(\log\lambda)^2\le(\log\delta)^2$. On $[1,\infty)$, there exists $C_{\epsilon}>0$ such that $	(\log\lambda)^2\le C_\varepsilon\,\lambda^{2\varepsilon},$
		so
		\[
		\int_{[1,\infty)}(\log\lambda)^2\,dE_{f,f}
		\;\le\;
		C_\varepsilon\int_{[1,\infty)}\lambda^{2\varepsilon}\,dE_{f,f}
		\;<\;\infty
		\]
		since $f\in H^{2\varepsilon}(M)$.  Thus $\int_{(0,\infty)}(\log\lambda)^2\,dE_{f,f}<\infty$.
		
		\noindent\textbf{(2) The case $\delta=0$.}
Choose positive spectral points tending to zero and pairwise disjoint open
intervals $J_k\subset(0,e^{-k^2})$ containing them. Such intervals can be
chosen inductively because zero is an accumulation point of the positive
spectrum. Each $E(J_k)$ is nonzero: otherwise a smaller open neighborhood
of the chosen spectral point would have zero spectral projection, contrary
to the definition of the spectral support; see
\cite[Proposition E.2(a)]{keller2021graphs}.
Choose unit vectors $f_k\in\operatorname{Ran}E(J_k)$ and set
\begin{equation}\label{fdls}
f=\sum_{k=1}^\infty k^{-1}f_k.
\end{equation}
The vectors are orthogonal, so the series converges in $L^2$.
Its spectral support lies in $(0,1)$; hence, for every $\varepsilon>0$,
\[
\int_{[0,\infty)}\lambda^{2\varepsilon}\,dE_{f,f}
\le\sum_{k=1}^\infty k^{-2}<\infty.
\]
On the other hand,
\[
\int_{(0,\infty)}|\log\lambda|^2\,dE_{f,f}
\ge\sum_{k=1}^\infty k^{-2}k^4=\infty.
\]
This proves the asserted noninclusion.
	\end{proof}
	
	\begin{remark}[Examples of spectral accumulation at zero]\label{zdsa}
		A basic class of examples satisfying the condition in part (ii) is given by the Euclidean space \(\mathbb{R}^n\).
	\end{remark}
	
	\begin{remark}[Spectral gap vs.\ essential spectrum]
		The quantity
		\(\delta=\inf\bigl(\sigma(-\Delta)\setminus\{0\}\bigr)\)
		measures the spectral gap above the zero‐eigenvalue. By contrast, $\inf\sigma_{\rm ess}(-\Delta)$
		is the bottom of the essential spectrum.  These two can coincide, as in Remark \ref{zdsa}. But they may differ in some cases, such as the Laplace–Beltrami on a closed connected Riemannian manifold, where $\delta>0$ and $\sigma_{\rm ess}\left(-\Delta\right)=\emptyset.$
	\end{remark}
	
	\begin{remark}
		On Euclidean space $\mathbb{R}^n$, the abstract spectral measure admits the following concrete Fourier descriptions with
		$$\widehat f(\xi)=(2\pi)^{-n/2}\int f(x)e^{-ix\cdot\xi}dx,$$
		i.e. for any Borel set $B\subset[0,\infty)$,
		$$
		E_{f,f}(B)
		=\int_{\{|\xi|^2\in B\}}|\widehat f(\xi)|^2\,d\xi.
		$$
		Define the spectral measure density
		$$
		F(\lambda)=E_{f,f}((0,\lambda])
		=\int_{|\xi|\le\sqrt\lambda}|\widehat f(\xi)|^2\,d\xi,
		$$
		then differentiating in $\lambda$ (in polar coordinates $\xi=r\omega$) gives
		$$
		dE_{f,f}(\lambda)
		=F'(\lambda)\,d\lambda
		=\frac12\,\lambda^{\frac n2-1}
		\int_{\mathbb{S}^{n-1}}\bigl|\widehat f(\sqrt\lambda\,\omega)\bigr|^2\,d\omega
		\;d\lambda.
		$$
		Hence	
		$$
		\int_{[0,\infty)}\lambda^2\,dE_{f,f}(\lambda)
		=\int_0^\infty\lambda^2\Bigl[\tfrac12\,\lambda^{\frac n2-1}
		\!\!\int_{\mathbb{S}^{n-1}}\!|\widehat f(\sqrt\lambda\,\omega)|^2\,d\omega\Bigr] 
		d\lambda
		=\int_{\mathbb{R}^n}|\xi|^4\,|\widehat f(\xi)|^2\,d\xi,
		$$
		which is exactly $\|(-\Delta)f\|_{L^2}^2$ in Fourier variables.
	\end{remark}
	
	It follows that the spectral Sobolev spaces coincide with the classical fractional Sobolev spaces on \(\mathbb{R}^n\), where fractional Sobolev spaces are given by
	$$
	H^s(\mathbb{R}^n)
	=\bigl\{f\in L^2(\mathbb{R}^n):\,(1+|\xi|^2)^{\frac s2}\,\widehat f(\xi)\in L^2(\mathbb{R}^n)\bigr\},
	$$
	with norm
	$$
	\|f\|_{H^s}^2
	=\int_{\mathbb{R}^n}(1+|\xi|^2)^s\,\bigl|\widehat f(\xi)\bigr|^2\,d\xi.
	$$
	and the logarithmic Sobolev space
	$$
	H^{\log}(\mathbb{R}^n)
	=\bigl\{f\in L^2(\mathbb{R}^n):\log\bigl(|\xi|^2\bigr)\,\widehat f(\xi)\in L^2(\mathbb{R}^n)\bigr\},
	$$
	with norm
	$$
	\|f\|_{H^{\log}}^2
	=\int_{\mathbb{R}^n}\bigl[1+\log^2\bigl(|\xi|^2\bigr)\bigr]\,\bigl|\widehat f(\xi)\bigr|^2\,d\xi.
	$$
	
	If \(f\in C_c^\infty(\mathbb{R}^n)\), then \(\widehat f(\xi)\) is a Schwartz‐class function and in particular decays faster than any power of \(|\xi|\).  Hence
	\[
	(1+|\xi|^2)^{\frac s2}\,\widehat f(\xi)\;\in\;L^2(\mathbb{R}^n)
	\quad\text{for every }s\ge0,
	\]
	so \(C_c^\infty(\mathbb{R}^n)\subset H^s(\mathbb{R}^n)\).  In fact \(C_c^\infty(\mathbb{R}^n)\) is dense in each \(H^s(\mathbb{R}^n)\).
	
	Similarly, since 
	\[
	\int_{|\xi|\le1}\log^2\bigl(|\xi|^2\bigr)\,\bigl|\widehat f(\xi)\bigr|^2\,d\xi
	\;\le\;
	|\mathbb S^{n-1}|\,\|\widehat f\|_{L^\infty}^2
	\int_{0}^{1}\!r^{n-1}\,\bigl|\log(r^2)\bigr|^2\,dr,
	\]
	then 
	\[
	\log\bigl(|\xi|^2\bigr)\,\widehat f(\xi)\;\in\;L^2(\mathbb{R}^n),
	\]
	and therefore \(C_c^\infty(\mathbb{R}^n)\subset H^{\log}(\mathbb{R}^n)\).  Again, \(C_c^\infty(\mathbb{R}^n)\) is dense in \(H^{\log}(\mathbb{R}^n)\).
	
	We now take a more intrinsic viewpoint to analyze the limits of the fractional Laplacian as \(s \to 0^+\) and \(s \to 1^-\), as well as its relation to the logarithmic Laplacian. This gives spectral proofs of the corresponding $L^2$ limits on
$\mathbb R^n$. It does not by itself prove all the pointwise or $L^p$
statements of \cite[Theorem 1.1]{chen2019dirichlet} and
\cite[Proposition 4.4]{dinezza2012hitchhiker}.

\begin{proof}[\textbf{Proof of
		Proposition~\ref{slx11}.}]
		By the spectral decomposition, we obtain
		\[
		(-\Delta)^s f
		=\int_{[0,\infty)}\lambda^s\,dE(\lambda)\,f
		=\int_{(0,\infty)}\lambda^s\,dE(\lambda)\,f.
		\]
		Since 
		\[
		f
		=E(\{0\})f+\int_{(0,\infty)}1\,dE(\lambda)\,f,
		\]
		we have
		\[
		(-\Delta)^s f
		-\bigl(f - E(\{0\})f\bigr)
		=\int_{(0,\infty)}(\lambda^s-1)\,dE(\lambda)\,f.
		\]
		By the spectral representation,
		\[
		\bigl\|(-\Delta)^s f - f + E(\{0\})f\bigr\|^2
		=\int_{(0,\infty)}(\lambda^s - 1)^2\,dE_{f,f}(\lambda).
		\]
		On $(0,\infty)$ we have $\lim_{s\to0^+}(\lambda^s-1)^2=0$, and for any fixed $\epsilon>0$ there is $C>0$ such that
		\[
		(\lambda^s - 1)^2 \;\le\; C\bigl(1+\lambda^{\epsilon}\bigr),
		\quad
		0<s<\frac{\epsilon}{2},
		\]
		while $f\in H^\epsilon\left(M\right)$ implies
		$\displaystyle\int_{[0,\infty)}(1+\lambda^{\epsilon})\,dE_{f,f}<\infty$.  By dominated convergence,
		\[
		\int_{(0,\infty)}(\lambda^s - 1)^2\,dE_{f,f}
		\;\xrightarrow{s\to0^+}\;0,
		\]
        which completes the proof.
	\end{proof}
	
	By \cite[Proposition E.2(b)]{keller2021graphs}, $P_0\ne0$ exactly
when zero is an eigenvalue. On a complete connected manifold without
boundary every $L^2$ harmonic function is constant. Indeed, testing the
harmonic equation with $\chi_R^2u$, where $\chi_R$ is a compactly supported
cutoff with $|\nabla\chi_R|\le C/R$, gives
\[
\int_M\chi_R^2|\nabla u|^2\le C R^{-2}\|u\|_2^2.
\]
Letting $R\to\infty$ proves the assertion; compare
\cite[Theorem 2.3]{caselli2024asymptotics}.
It follows that on a complete connected manifold without boundary,
\[
\ker_{L^2}(-\Delta)=
\begin{cases}
\{0\},&\operatorname{Vol}(M)=\infty,\\
\operatorname{span}\{\mathbf 1\},&\operatorname{Vol}(M)<\infty.
\end{cases}
\]
In the infinite-volume case, $P_0f=0$ for every $f\in L^2(M)$, even
though zero may still belong to the continuous spectrum. In the
finite-volume case,
\[
P_0f=\left(\frac1{\operatorname{Vol}(M)}
\int_M f\,d\mathrm{vol}\right)\mathbf 1,
\qquad
Qf=f-\left(\frac1{\operatorname{Vol}(M)}
\int_M f\,d\mathrm{vol}\right)\mathbf 1.
\]
Thus $P_0f=0$ is equivalent to the zero-mean condition only in the
finite-volume case. Finite-volume manifolds may be noncompact, so
noncompactness alone does not imply $P_0=0$.
In both volume regimes we retain the same domain $H^{\log}(M)$:
the spectral integral over $(0,\infty)$ imposes no restriction on
$P_0f$, and the logarithmic operator maps this component to zero.
The same projection formula holds on a connected compact manifold with
Neumann boundary conditions.

\begin{proof}[\textbf{Proof of
		Proposition~\ref{slx22}.}]
		For \(f\in H^{2}\left(M\right)\), the spectral representation gives
		\[
		\bigl\|(-\Delta)^{s}f+\Delta f\bigr\|_{L^{2}}^{2}
		=\int_{0}^{\infty}\bigl(\lambda^{s}-\lambda\bigr)^{2}\,dE_{f,f}(\lambda).
		\]
		Note that \(\lambda^{s}\to\lambda\) for each \(\lambda\ge0\) as
		\(s\uparrow1\).  To apply dominated convergence, note that
		\[\left(\lambda^s-\lambda\right)^2\le c\left(1+\lambda^2\right),s\in \left(0,1\right),\lambda\in \left(0,\infty\right).\]
		Hence, 
		\[
		\lim_{s\to1^{-}}
		\int_{0}^{\infty}\bigl(\lambda^{s}-\lambda\bigr)^{2}\,dE_{f,f}(\lambda)=0,
		\]
		i.e.\ \((-\Delta)^{s}f\to-\Delta f\) in \(L^{2}(M,\mu)\) as \(s\uparrow1\).
	\end{proof}

\begin{proof}[\textbf{Proof of
		Proposition~\ref{slx33}.}]
For $\lambda>0$, set $h_s(\lambda)=(\lambda^s-1)/s-\log\lambda$.
The mean value theorem gives
$h_s(\lambda)=(\lambda^\theta-1)\log\lambda$ for some $0<\theta<s$.
Fix $0<s_0<\varepsilon/2$ and restrict $0<s<s_0$.
For $0<\lambda\le1$, $|h_s(\lambda)|\le|\log\lambda|$;
for $\lambda\ge1$,
\[
|h_s(\lambda)|^2\le C\lambda^{2s_0}(\log\lambda)^2
\le C_\varepsilon\lambda^\varepsilon.
\]
These bounds are integrable against $E_{f,f}$ because
$f\in H^\varepsilon(M)\cap H^{\log}(M)$. Since $h_s(\lambda)\to0$,
the spectral theorem and dominated convergence yield
\[
\left\|\frac{(-\Delta)^s-I+P_0}{s}f-\log(-\Delta)f\right\|_2^2
=\int_{(0,\infty)}|h_s(\lambda)|^2\,dE_{f,f}(\lambda)\longrightarrow0,
\]
which completes the proof.
\end{proof}

	\subsection{Bochner integral representation}
	
	We have already defined the fractional and logarithmic Laplacians via the spectral measure. However, for general Riemannian manifolds, the spectral measure can be rather abstract and not well suited for detailed analytical investigations. In what follows, we combine the theory of operator semigroups with  functional calculus to derive an explicit Bochner integral representation for both the fractional and logarithmic Laplacians, and to establish fundamental properties.
	
	In particular, to the best of our knowledge, the Bochner representation for the logarithmic Laplacian presented here is new and has not appeared previously in the literature. We interpret the time integrals on their natural spectral domains as strong improper integrals and treat the zero eigenspace separately. For every $t>0,$ put $\varphi_{t}(\lambda)=e^{-t\lambda}$.  
	Then $e^{t\Delta}:=\varphi_{t}(-\Delta)$
	is the strongly continuous contraction heat semigroup on $L^{2}(M,\mu)$ with generator $\Delta$. For $f\in \operatorname{Dom}\left(-\Delta\right)$, $e^{t\Delta}f$ is a unique solution of $\partial_t u_t=\Delta u_t$ with initial value $u_0=f$.
	
	For $0<s<1$, the scalar identity
\[
\lambda^s=\frac{s}{\Gamma(1-s)}
\int_0^\infty(1-e^{-t\lambda})t^{-1-s}\,dt,
\qquad\lambda\ge0,
\]
leads to the following formula.
\begin{equation}\label{eq:balakrishnan}
(-\Delta)^sf=\frac{s}{\Gamma(1-s)}
\int_0^\infty(f-e^{t\Delta}f)t^{-1-s}\,dt,
\qquad f\in H^{2s}(M).
\end{equation}
The integral follows the strong improper convention fixed in the introduction.

\begin{proposition}\label{fslps}
Formula~\eqref{eq:balakrishnan} holds for every $f\in H^{2s}(M)$.
For every $g\in L^2(M)$, the corresponding scalar iterated integrals
may be interchanged.
\end{proposition}
\begin{proof}
Put $c_s=s/\Gamma(1-s)$ and
$m_{\varepsilon,R}(\lambda)=c_s\int_\varepsilon^R
(1-e^{-t\lambda})t^{-1-s}\,dt$.
The integrand is nonnegative and
$0\le m_{\varepsilon,R}(\lambda)\le\lambda^s$.
Moreover, $m_{\varepsilon,R}(\lambda)\to\lambda^s$ as the endpoints
tend to zero and infinity. Spectral dominated convergence gives
\begin{equation}\label{bochner}
\|m_{\varepsilon,R}(-\Delta)f-(-\Delta)^sf\|_2^2
=\int_{[0,\infty)}|m_{\varepsilon,R}(\lambda)-\lambda^s|^2
\,dE_{f,f}(\lambda)\longrightarrow0.
\end{equation}
On a finite time interval the vector integrand is continuous and its
Bochner integral equals $m_{\varepsilon,R}(-\Delta)f$.
For scalar pairings, the spectral Cauchy--Schwarz inequality gives
\[
\int_{[0,\infty)}\lambda^s\,d|E_{f,g}|(\lambda)
\le\|(-\Delta)^sf\|_2\|g\|_2<\infty,
\]
so scalar Fubini is justified as well.
\end{proof}

The heat-semigroup formula is classical; see \cite{kwasnicki2017ten}.
The logarithmic version follows from Frullani's identity, which we record
with a proof.

	\begin{lemma}[Scalar integral representation]\label{scalr}
		For each \(\lambda>0\),
		\[
		I(\lambda)
		:=
		\int_{0}^{\infty}\frac{e^{-t}-e^{-\lambda t}}{t}\,dt
		=\log\lambda.
		\]
	\end{lemma}
	
	\begin{proof}
		
		For $\lambda >1,$ note that
		\[\frac{\partial}{\partial \lambda}\left(\frac{e^{-t}-e^{-\lambda t}}{t}\right)\le e^{-t},\]
		by the dominated convergence theorem,
		\[\frac{d}{d\lambda}I\left(\lambda\right)=\int_0^{\infty}e^{-\lambda t}dt=\frac{1}{\lambda}.\]
		Thus, $I\left(\lambda\right)=\log \lambda +c.$ It is not hard to see that $c=0.$ For $\lambda \in \left(0,1\right),$ set $t=\frac{s}{\lambda},$ then 
		\[I\left(\lambda\right)=\int_0^\infty \frac{e^{-\frac{s}{\lambda}}-e^{-s}}{s}ds=\log\lambda.\]
		Thus, we complete the proof.
	\end{proof}
	
	\medskip
	
	Combining the scalar lemma with the spectral theorem, we give the proof of Theorem \ref{logboch11}.

\begin{proof}[\textbf{Proof of
		Theorem~\ref{logboch11}.}]
For $0<\varepsilon<R$, define $b_{\varepsilon,R}(0)=0$ and, for $\lambda>0$,
\[
b_{\varepsilon,R}(\lambda)
=\int_\varepsilon^R\frac{e^{-t}-e^{-t\lambda}}{t}\,dt.
\]
For a fixed $\lambda>0$ the scalar integrand has constant sign. Therefore
Lemma~\ref{scalr} implies
\[
|b_{\varepsilon,R}(\lambda)|\le|\log\lambda|,
\qquad b_{\varepsilon,R}(\lambda)\longrightarrow\log\lambda.
\]
Consequently, for $f\in H^{\log}(M)$,
\[
\|b_{\varepsilon,R}(-\Delta)f-\log(-\Delta)f\|_2^2
=\int_{(0,\infty)}|b_{\varepsilon,R}(\lambda)-\log\lambda|^2
\,dE_{f,f}(\lambda)\longrightarrow0.
\]
On $[\varepsilon,R]$ the vector integrand is continuous, and spectral
calculus identifies its Bochner integral with
$b_{\varepsilon,R}(-\Delta)f$. Since $e^{t\Delta}P_0=P_0$,
\[
e^{-t}Qf-e^{t\Delta}Qf
=e^{-t}f-e^{t\Delta}f+(1-e^{-t})P_0f.
\]
This proves both forms of the theorem. In particular, on a closed
manifold the correction removes the divergent contribution of constants.
The argument proves strong convergence of the truncated integrals;
it does not assert absolute integrability of their $L^2$ norms on
$(0,\infty)$.
\end{proof}

\subsection{Logarithmic Laplacian on \texorpdfstring{\(\mathbb{R}^n\)}{Euclidean n-space}}

Since $\mathbb R^n$ is complete and has infinite volume, $P_0=0$ on
$L^2(\mathbb R^n)$. Hence $Q=I$ in Theorem~\ref{logboch11}; no
zero-mean condition on $f$ is required.
	
	We now turn to the classical setting \(M = \mathbb{R}^n\). For the fractional Laplacian, it is well known that the Bochner integral representation is equivalent to the pointwise kernel formula involving singular integrals; see \cite[Theorem 1.1]{kwasnicki2017ten}. A natural question is whether, in the case of the logarithmic Laplacian, the Bochner integral definition we have just introduced,
	\[
	\log(-\Delta)\,f(x)
	= \int_0^\infty \frac{e^{-t}f(x) - e^{t\Delta}f(x)}{t}\,dt,
	\]
	can recover the pointwise kernel representation of \(L_{\!\Delta}\) on \(\mathbb{R}^n\) given by H.~Chen and T.~Weth in \cite{chen2019dirichlet}. Theorem~\ref{logrn11} gives a positive answer.
	
	In \(\mathbb{R}^n\), the semigroup \(e^{t\Delta}f\) corresponds to the Gauss semigroup and is given by
	\[
	e^{t\Delta}f(x) = \frac{1}{(4\pi t)^{n/2}} \int_{\mathbb{R}^n} e^{-\frac{|x - y|^2}{4t}}\, f(y)\,dy.
	\]
	
	To prove Theorem~\ref{logrn11}, we first present the following lemmas, which play a crucial role in the argument.

	\begin{lemma}\label{euler}
		The following identity holds:
		\[
		\int_{0}^{1}\frac{e^{-t}-1}{t}\,dt
		+\int_{1}^{\infty}\frac{e^{-t}}{t}\,dt
		=-\gamma,
		\]
		where \(\gamma\) is the Euler–Mascheroni constant.
	\end{lemma}

	\begin{proof}
		Set
		\[
		I_1 := \int_{0}^{1}\frac{e^{-t}-1}{t}\,dt,\qquad
		I_2 := \int_{1}^{\infty}\frac{e^{-t}}{t}\,dt.
		\]
		We will show \(I_1+I_2=-\gamma\) by relating them to the well‐known formula
		\[
		\int_{0}^{\infty}e^{-t}\log t\,dt = \Gamma'(1) = -\gamma.
		\]
		Integrate by parts with 
		\(\displaystyle u=e^{-t}-1,\; dv=\frac{dt}{t}\):
		\[
		\begin{aligned}
			I_1
			= \Bigl[(e^{-t}-1)\log t\Bigr]_{0}^{1}
			-\int_{0}^{1}\log t\,(-e^{-t})\,dt = \int_{0}^{1}e^{-t}\log t\,dt.
		\end{aligned}
		\]
		Similarly, we have
		\[
		\begin{aligned}
			I_2
			&= \Bigl[e^{-t}\log t\Bigr]_{1}^{\infty}
			-\int_{1}^{\infty}(-e^{-t})\log t\,dt \\
			&= \int_{1}^{\infty}e^{-t}\log t\,dt.
		\end{aligned}
		\]
		Adding the two pieces gives
		\[
		I_1 + I_2
		= \int_{0}^{1}e^{-t}\log t\,dt + \int_{1}^{\infty}e^{-t}\log t\,dt
		= \int_{0}^{\infty}e^{-t}\log t\,dt = -\gamma.
		\]
		This completes the proof.
	\end{proof}
	
	\begin{lemma}\label{fhzdk}
		The following identity holds: $$\int_{\frac{1}{4}}^{\infty}\int_{s}^{\infty}\frac{t^{\frac{n}{2}-1}e^{-t}}{2s}dtds-\int_0^{\frac{1}{4}}\int_{0}^{s}\frac{t^{\frac{n}{2}-1}e^{-t}}{2s}dtds=\frac{\Gamma^{\prime}\left(\frac{n}{2}\right)}{2}+\Gamma\left(\frac{n}{2}\right)\ln2.$$
	\end{lemma}
	
	\begin{proof}
		Set
		\[
		I = \int_{\frac14}^{\infty}\int_{s}^{\infty}\frac{t^{\frac n2-1}e^{-t}}{2s}\,dt\,ds
		\;-\;
		\int_{0}^{\frac14}\int_{0}^{s}\frac{t^{\frac n2-1}e^{-t}}{2s}\,dt\,ds.
		\]
		We split into two regions and interchange the order of integration by Fubini’s theorem.
		
		1. \emph{Region \(s\ge\frac14,\;t\ge s\).}  Here
		\[
		\int_{\frac14}^{\infty}\int_{s}^{\infty}\frac{t^{\frac n2-1}e^{-t}}{2s}\,dt\,ds
		=\int_{\frac14}^{\infty} t^{\frac n2-1}e^{-t}
		\Bigl(\int_{\frac14}^{t}\frac{1}{2s}\,ds\Bigr)dt=\frac12	\int_{\frac14}^{\infty}t^{\frac n2-1}e^{-t}\,\ln(4t)\,dt.
		\]
		
		2. \emph{Region \(0\le s\le\frac14,\;0\le t\le s\).}  Here
		\[
		\int_{0}^{\frac14}\int_{0}^{s}\frac{t^{\frac n2-1}e^{-t}}{2s}\,dt\,ds
		=\int_{0}^{\frac14} t^{\frac n2-1}e^{-t}
		\Bigl(\int_{t}^{\frac14}\frac{1}{2s}\,ds\Bigr)dt=-\tfrac12\int_{0}^{\frac14}t^{\frac n2-1}e^{-t}\,\ln(4t)\,dt.
		\]
		
		Combining both gives
		\[
		I
		=\frac12\int_{0}^{\infty}t^{\frac n2-1}e^{-t}\ln(4t)\,dt=\frac{1}{2}\int_{0}^{\infty}t^{\frac n2-1}e^{-t}\ln t\,dt
		\;+\;\ln 2\int_{0}^{\infty}t^{\frac n2-1}e^{-t}\,dt.
		\]
		By the well‐known Gamma‐function identities,
		\[
		\int_{0}^{\infty}t^{a-1}e^{-t}\,dt
		=\Gamma(a),
		\quad
		\int_{0}^{\infty}t^{a-1}e^{-t}\ln t\,dt
		=\Gamma(a)\,\psi(a),
		\]
		where \(\psi(a)=\Gamma'(a)/\Gamma(a)\).  Taking \(a=\tfrac n2\) yields
		\[
		I
		=\frac12\,\Gamma\Bigl(\tfrac n2\Bigr)\Bigl[\psi\Bigl(\tfrac n2\Bigr)+\ln4\Bigr]
		=\frac{\Gamma'\bigl(\tfrac n2\bigr)}{2}
		\;+\;\Gamma\Bigl(\tfrac n2\Bigr)\ln2,
		\]
		as claimed.
	\end{proof}

		Combining Lemma \ref{euler} and Lemma \ref{fhzdk}, we give the proof of Theorem \ref{logrn11}.

\begin{proof}[\textbf{Proof of
		Theorem~\ref{logrn11}.}]
		For $0<\beta_0\le\min\{\beta,1\}$ the Gaussian kernel gives
\[
\int_{\mathbb R^n}p_t(x,y)|f(x)-f(y)|\,dy\le C t^{\beta_0/2},
\quad 0<t\le1.
\]
Indeed, $f$ is globally $\beta_0$-H\"older and the Gaussian
$\beta_0$-moment equals a constant times $t^{\beta_0/2}$.
The same argument, symmetry, and compact support give
$\|f-P_tf\|_2\le Ct^{\beta_0/2}$, while
$\|P_tf\|_2\le Ct^{-n/4}\|f\|_1$ for $t\ge1$.
Thus the time integral converges absolutely in $L^2$; spectral Fatou
followed by Theorem~\ref{logboch11} shows that $f\in H^{\log}$ and
identifies its value with the spectral operator. Pointwise convergence
follows from the displayed bound and
$|P_tf(x)|\le Ct^{-n/2}\|f\|_1$. We now compute this representative.
We split the time integral
		\[
		\int_{0}^{\infty}\frac{e^{-t}f - e^{t\Delta}f}{t}\,dt
		=\int_{0}^{1}\frac{e^{-t}f - e^{t\Delta}f}{t}\,dt+\int_{1}^{\infty}\frac{e^{-t}f - e^{t\Delta}f}{t}\,dt.
		\]
		\textbf{(1) Short–time contribution \(0<t<1\).}
		Write the heat semigroup on \(\mathbb R^{n}\) as
		\[
		(e^{t\Delta}f)(x)=\frac{1}{(4\pi t)^{n/2}}
		\int_{\mathbb R^{n}}e^{-\frac{|x-y|^{2}}{4t}}\,f(y)\,dy.
		\]
		Hence
		\[
		e^{-t}f(x)-e^{t\Delta}f(x)
		=\frac{1}{(4\pi t)^{n/2}}
		\int_{\mathbb R^{n}}e^{-\frac{|x-y|^{2}}{4t}}
		\bigl(e^{-t}f(x)-f(y)\bigr)\,dy.
		\]
		Set \(y=x+\sqrt{t}\,z\), then
		\[
		e^{-t}f(x)-e^{t\Delta}f(x)
		=\frac{1}{(4\pi)^{n/2}}
		\int_{\mathbb R^{n}}e^{-|z|^{2}/4}
		\bigl(e^{-t}f(x)-f(x+\sqrt{t}\,z)\bigr)\,dz.
		\]
		Insert this into the time–integral and divide by \(t\):
		\[
		\begin{aligned}
			\int_{0}^{1}\frac{e^{-t}f(x)-e^{t\Delta}f(x)}{t}\,dt
			&=\frac{1}{(4\pi)^{n/2}}
			\int_{\mathbb R^{n}}\!e^{-|z|^{2}/4}
			\int_{0}^{1}\frac{e^{-t}f(x)-f(x+\sqrt{t}\,z)}{t}\,dt\,dz.
		\end{aligned}
		\]
		Note that
		\[
		\int_{0}^{1}\frac{e^{-t}f(x)-f(x+\sqrt{t}z)}{t}\,dt
		=\int_{0}^{1}\frac{f(x)-f(x+\sqrt{t}z)}{t}\,dt
		+f(x)\!\int_{0}^{1}\frac{e^{-t}-1}{t}\,dt.
		\]
		For the second term, note that
		$$
		\int_{\mathbb{R}^{n}}e^{-\lvert z\rvert^{2}/4}\,dz
		=(4\pi)^{n/2}.
		$$	
		For the first term set \(t=r^{2}\), then
		\[
		\int_{0}^{1}\frac{f(x)-f(x+\sqrt{t}z)}{t}\,dt
		=\int_{0}^{1}\frac{f(x)-f(x+r z)}{r^{2}}\,2r\,dr
		=2\int_{0}^{1}\frac{f(x)-f(x+r z)}{r}\,dr.
		\]
		The bound $|f(x)-f(x+rz)|\le C r^{\beta_0}|z|^{\beta_0}$
and the finite Gaussian moment justify Fubini and give
		\[
		2\int_{\mathbb R^{n}}\!e^{-|z|^{2}/4}
		\int_{0}^{1}\frac{f(x)-f(x+r z)}{r}\,dr\,dz
		=\int_{0}^{1}\frac{2}{r}
		\int_{\mathbb R^{n}} e^{-|z|^{2}/4}
		\bigl(f(x)-f(x+r z)\bigr)\,dz\,dr.
		\]
		Set \(w=r z\).  Then \(dz=r^{-n}dw\) and
		\[
		\int_{\mathbb R^{n}} e^{-|z|^{2}/4}\bigl(f(x)-f(x+r z)\bigr)\,dz
		=r^{-n}\int_{\mathbb R^{n}}e^{-|w|^{2}/(4r^{2})}\bigl(f(x)-f(x+w)\bigr)\,dw.
		\]
		Thus,
		\[
		\begin{aligned}
			\int_{0}^{1}\frac{e^{-t}f(x)-e^{t\Delta}f(x)}{t}\,dt
			&=f(x)\!\int_{0}^{1}\frac{e^{-t}-1}{t}\,dt\\
 &\quad+\frac{1}{(4\pi)^{n/2}}
			\int_{0}^{1}\frac{2}{t^{n+1}}\int_{\mathbb R^{n}}\!e^{\frac{-|y|^{2}}{4t^2}}\left(f(x)-f(x+y)\right)\,dy\,dt.
		\end{aligned}
		\]
		Set $u=\frac{|y|^{2}}{4t^2}$, then we calculate that 
		\[\frac{1}{(4\pi)^{n/2}}
		\int_{0}^{1}\frac{2}{t^{n+1}}e^{\frac{-|y|^{2}}{4t^2}}dt=\pi^{-\frac{n}{2}}|y|^{-n}\int_{\frac{|y|^2}{4}}^{\infty}t^{\frac{n}{2}-1}e^{-t}dt,\]
		so 
		\[\begin{aligned}
&\frac{1}{(4\pi)^{n/2}}
		\int_{0}^{1}\frac{2}{t^{n+1}}\int_{\mathbb R^{n}}\!e^{\frac{-|y|^{2}}{4t^2}}\left(f(x)-f(x+y)\right)\,dy\,dt\\
 &=\pi^{-\frac{n}{2}}\int_{\mathbb R^{n}}\int_{\frac{|y|^2}{4}}^{\infty}t^{\frac{n}{2}-1}e^{-t}dt\frac{\left(f(x)-f(x+y)\right)}{|y|^n}dy.
\end{aligned}\]
		Note that
		\[\int_{\frac{|y|^2}{4}}^{\infty}t^{\frac{n}{2}-1}e^{-t}dt=\Gamma\left(\frac{n}{2}\right)-\int_{0}^{\frac{|y|^2}{4}}t^{\frac{n}{2}-1}e^{-t}dt,\]
		then
		\[\begin{aligned}
			\int_{0}^{1}\frac{e^{-t}f(x)-e^{t\Delta}f(x)}{t}\,dt
			=&f(x)\!\int_{0}^{1}\frac{e^{-t}-1}{t}\,dt+c_{n}\int_{B_1(0)}\frac{\left(f(x)-f(x+y)\right)}{|y|^n}dy\\-&\pi^{-\frac{n}{2}}\int_{B_1(0)}\int_{0}^{\frac{|y|^2}{4}}t^{\frac{n}{2}-1}e^{-t}dt\frac{\left(f(x)-f(x+y)\right)}{|y|^n}dy\\+&\pi^{-\frac{n}{2}}\int_{\mathbb R^{n}\setminus B_1(0)}\int_{\frac{|y|^2}{4}}^{\infty}t^{\frac{n}{2}-1}e^{-t}dt\frac{\left(f(x)-f(x+y)\right)}{|y|^n}dy.
		\end{aligned}\]
		\medskip
		\noindent\textbf{(2) Long–time contribution \(\boldsymbol{t\ge1}\).} We start from
		\[\int_{1}^{\infty}\frac{e^{-t}f - e^{t\Delta}f}{t}\,dt=f\left(x\right)\int_1^\infty \frac{e^{-t}}{t}dt-\int_1^{\infty}\int_{\mathbb R^{n}}\frac{1}{t}\frac{1}{(4\pi t)^{n/2}}
		e^{-\frac{|x-y|^{2}}{4t}}\,f(y)\,dydt.\]
		For the second term,
		\[\int_1^{\infty}\int_{\mathbb R^{n}}\frac{1}{t}\frac{1}{(4\pi t)^{n/2}}
		e^{-\frac{|x-y|^{2}}{4t}}\,f(y)\,dy\,dt=\int_1^{\infty}\int_{\mathbb R^{n}}\frac{1}{t}\frac{1}{(4\pi t)^{n/2}}
		e^{-\frac{|y|^{2}}{4t}}\,f(x+y)\,dydt.\]
		Similarly, we can calculate that
		\[\begin{aligned}
			\int_1^{\infty}\frac{1}{t}\frac{1}{(4\pi t)^{n/2}}
			e^{-\frac{|y|^{2}}{4t}}dt=&\pi^{-\frac{n}{2}} |y|^{-n} \int_0^{|y|^2/4} t^{n/2 - 1} e^{-t} \, dt\\=&\pi^{-\frac{n}{2}} |y|^{-n}\left(\Gamma\left(\frac{n}{2}\right)-\int_{\frac{|y|^2}{4}}^{\infty}t^{\frac{n}{2}-1}e^{-t}dt\right).
		\end{aligned}\]
		Hence,
		\[\begin{aligned}
			\int_{1}^{\infty}\frac{e^{-t}f - e^{t\Delta}f}{t}\,dt=&f\left(x\right)\int_1^\infty \frac{e^{-t}}{t}dt-\pi^{-\frac{n}{2}}\int_{B_1(0)}\int_{0}^{\frac{|y|^2}{4}}t^{\frac{n}{2}-1}e^{-t}dt\frac{f(x+y)}{|y|^n}dy\\-&c_{n}\int_{\mathbb{R}^{n}\setminus B_1(0)}\frac{f(x+y)}{|y|^n}dy+\pi^{-\frac{n}{2}}\int_{\mathbb{R}^{n}\setminus B_1(0)}\int_{\frac{|y|^2}{4}}^{\infty}t^{\frac{n}{2}-1}e^{-t}dt\frac{f(x+y)}{|y|^n}dy
		\end{aligned}\]
		Combining (1) and (2), by Lemma \ref{euler}, we obtain that
		\[\begin{aligned}
			\int_{0}^{\infty}\frac{e^{-t}f - e^{t\Delta}f}{t}\,dt=&c_{n}\int_{B_{1}(x)}\frac{f(x)-f(y)}{|x-y|^{n}}\,dy
			-c_{n}\int_{\mathbb{R}^{n}\setminus B_{1}(x)}
			\frac{f(y)}{|x-y|^{n}}\,dy-\gamma \:f(x)\\+&\pi^{-\frac{n}{2}}\left\{\int_{\mathbb{R}^{n}\setminus B_1(0)}\int_{\frac{|y|^2}{4}}^{\infty}\frac{t^{\frac{n}{2}-1}e^{-t}}{|y|^n}dtdy-\int_{B_1(0)}\int_{0}^{\frac{|y|^2}{4}}\frac{t^{\frac{n}{2}-1}e^{-t}}{|y|^n}dtdy\right\}f(x).
		\end{aligned}\]
		For the last term, Lemma~\ref{fhzdk} gives
		\[\begin{aligned}
			&	\int_{\mathbb{R}^{n}\setminus B_1(0)}\int_{\frac{|y|^2}{4}}^{\infty}\frac{t^{\frac{n}{2}-1}e^{-t}}{|y|^n}dtdy-\int_{B_1(0)}\int_{0}^{\frac{|y|^2}{4}}\frac{t^{\frac{n}{2}-1}e^{-t}}{|y|^n}dtdy\\=&|\mathbb{S}^{n-1}| \int_1^{\infty}\int_{\frac{r^2}{4}}^{\infty}\frac{t^{\frac{n}{2}-1}e^{-t}}{r}dtdr-|\mathbb{S}^{n-1}|\int_0^{1}\int_{0}^{\frac{r^2}{4}}\frac{t^{\frac{n}{2}-1}e^{-t}}{r}dtdr\\=&|\mathbb{S}^{n-1}|\int_{\frac{1}{4}}^{\infty}\int_{s}^{\infty}\frac{t^{\frac{n}{2}-1}e^{-t}}{2s}dtds-|\mathbb{S}^{n-1}|\int_0^{\frac{1}{4}}\int_{0}^{s}\frac{t^{\frac{n}{2}-1}e^{-t}}{2s}dtds\\=&|\mathbb{S}^{n-1}|\frac{\Gamma^{\prime}\left(\frac{n}{2}\right)}{2}+|\mathbb{S}^{n-1}|\Gamma\left(\frac{n}{2}\right)\ln2.
		\end{aligned}\]
		Thus,
		\[\pi^{-\frac{n}{2}}\left\{\int_{\mathbb{R}^{n}\setminus B_1(0)}\int_{\frac{|y|^2}{4}}^{\infty}\frac{t^{\frac{n}{2}-1}e^{-t}}{|y|^n}dtdy-\int_{B_1(0)}\int_{0}^{\frac{|y|^2}{4}}\frac{t^{\frac{n}{2}-1}e^{-t}}{|y|^n}dtdy\right\}=2\ln 2+\psi\!\bigl(\tfrac{n}{2}\bigr).\]
		Hence, we complete the proof.
		\end{proof}

	\section{Logarithmic Laplacian on Riemannian Manifolds}\label{geometric-settings}
	
	In this section we derive explicit representations of both the fractional and logarithmic Laplacians in two geometric settings: first on compact Riemannian manifolds (with or without boundary), where the spectrum of the Laplace Beltrami operator is purely discrete; and then on noncompact complete Riemannian manifolds, where a general spectral measure allows both point and continuous spectrum. These formulas unify the treatment of nonlocal operators across all Riemannian manifolds.
	
	\subsection{Compact Riemannian manifolds}
	
	Due to the compactness of the manifold, the spectrum of the Laplace Beltrami operator is discrete under all standard settings, including closed manifolds as well as manifolds with Dirichlet or Neumann boundary conditions. As a consequence, the associated spectral measure consists of a pure point spectrum, and the definitions of both the fractional and logarithmic Laplacians become particularly simple. In the following, we restrict our attention to the case of closed manifolds.
	
	Let $(M,g)$ be a connected, closed Riemannian manifold. The Laplace--Beltrami operator
	\[
	-\Delta\colon C^{\infty}(M)\subset L^{2}(M)\to L^{2}(M)
	\]
	is essentially self-adjoint, and its closure is the unique self-adjoint extension in $L^2\left(M\right).$
	
	Since $M$ is closed and connected, $-\Delta$ has purely discrete spectrum
	\[
	0=\lambda_0<\lambda_1\le\lambda_2\le\cdots,\quad
	\lambda_j\to\infty,
	\]
	with a real-valued orthonormal eigenbasis $\{\varphi_j\}_{j=0}^{\infty}\subset C^{\infty}(M)$:
	\[
	-\Delta\varphi_j=\lambda_j\varphi_j,
	\qquad
	\langle\varphi_j,\varphi_k\rangle_{L^{2}}=\delta_{jk}.
	\]
	
	The heat semigroup $e^{t\Delta}$ is given pointwise by
	\[
	(e^{t\Delta}f)(x)=\int_{M}p_t(x,y)f(y)\,d\mathrm{vol}(y),
	\quad t>0,\;x\in M,
	\]
	where the heat kernel has the eigen‐expansion
	\[
	p_t(x,y)=\sum_{j=0}^{\infty}e^{-t\lambda_j}\varphi_j(x)\varphi_j(y).
	\]
	
	Let \(E(\lambda)\) be the unique projection–valued spectral measure of \(-\Delta\). For every \(f,g \in L^2(M)\), we have
	\[
	\langle E((\alpha,\beta])f, g\rangle
	= \sum_{\lambda_j \in (\alpha,\beta]} \langle f, \varphi_j \rangle \langle \varphi_j, g \rangle,
	\qquad -\infty < \alpha < \beta < \infty,
	\]
	where \(\{\varphi_j\}_{j=0}^{\infty}\) forms a complete orthonormal basis of eigenfunctions of \(-\Delta\) corresponding to eigenvalues \(0 = \lambda_0 < \lambda_1 \le \lambda_2 \le \cdots \to \infty\). The spectral calculus then yields, for any Borel measurable function \(\Phi : [0,\infty) \to \mathbb{R}\),
	\[
	\Phi(-\Delta)f
	= \int_{[0,\infty)} \Phi(\lambda)\, dE(\lambda)\, f
	= \sum_{j=0}^\infty \Phi(\lambda_j)\, \langle f, \varphi_j \rangle \varphi_j.
	\]
	
	In particular, for \(s\in (0,1)\), the fractional Laplacian is given by
	\[
	(-\Delta)^s f
	:= \int_0^{\infty} \lambda^s\, dE(\lambda)\, f
	= \sum_{j=0}^\infty \lambda_j^s\, \langle f, \varphi_j \rangle \varphi_j,
	\]
	and the corresponding fractional Sobolev space is
\[
H^{2s}(M)=\left\{f=\sum_{j=0}^\infty c_j\varphi_j:
\sum_{j=0}^\infty(1+\lambda_j)^{2s}|c_j|^2<\infty\right\}.
\]
Its norm is
\[
\|f\|_{H^{2s}(M)}^2=\sum_{j=0}^\infty(1+\lambda_j)^{2s}|c_j|^2,
\qquad c_j=\langle f,\varphi_j\rangle.
\]
This norm is equivalent to the graph norm with weights $1+\lambda_j^{2s}$.

	Similarly, the logarithmic Laplacian on a closed Riemannian manifold is defined by
	\[
	\log(-\Delta)\, f
	:= \int_{(0,\infty)} \log \lambda\, dE(\lambda)\, f
	= \sum_{j=1}^\infty (\log \lambda_j)\, \langle f, \varphi_j \rangle \varphi_j.
	\]
Here $\varphi_0=\operatorname{Vol}(M)^{-1/2}\mathbf 1$ may be chosen,
so for $f=\sum_{j=0}^\infty c_j\varphi_j$ we have
$P_0f=c_0\varphi_0$ and $Qf=\sum_{j=1}^\infty c_j\varphi_j$.
Theorem~\ref{logboch11} therefore gives exactly the same operator:
\[
\int_0^\infty\frac{e^{-t}Qf-e^{t\Delta}Qf}{t}\,dt
=\sum_{j=1}^\infty(\log\lambda_j)c_j\varphi_j,
\qquad f\in H^{\log}(M).
\]
This identity follows by spectral convergence in $L^2$, as in the proof
of the theorem. In particular, $c_0$ is unrestricted and
$\log(-\Delta)\mathbf 1=0$.
By Proposition~\ref{logarith}, the associated logarithmic Sobolev space is
	\[
	H^{\log}(M)
	= \left\{ f = \sum_{j=0}^\infty c_j \varphi_j \in L^2(M)\;\Big|\; \sum_{j=1}^\infty \bigl(1 + (\log \lambda_j)^2 \bigr)\, |c_j|^2 < \infty \right\},
	\]
	\[
	\|f\|_{H^{\log}(M)}^2
	= \sum_{j=0}^\infty  |c_j|^2+\sum_{j=1}^\infty  (\log \lambda_j)^2|c_j|^2,
	\quad c_j = \langle f, \varphi_j \rangle_{L^2(M)}.
	\]
	Thus, \(\log(-\Delta)\,f\) is well defined in \(L^2(M)\) for all \(f \in H^{\log}(M)\).

	\subsection{Noncompact Riemannian manifolds}
	
	We now turn to the setting of noncompact Riemannian manifolds, where we aim to define both the fractional and logarithmic Laplacians. In this context, the spectral measure associated with the Laplace Beltrami operator becomes significantly more intricate and depends heavily on the geometry of the underlying manifold. As a result, it is not feasible to work directly with the spectral expansion as in the compact case.
	
	Let \((M,g)\) be a  noncompact, complete and connected Riemannian manifold
	of dimension \(n\).
	Denote the Laplace Beltrami operator by
	\[
	-\Delta\colon C^\infty_{c}(M)\;\subset\;L^{2}(M)\longrightarrow L^{2}(M).
	\]
	
	\begin{theorem}[Gaffney--Chernoff]\cite{gaffney1954heat,chernoff1973essential}\label{gaff}
		If \((M,g)\) is complete, the operator \(-\Delta\)
		defined on \(C^\infty_{c}(M)\) is essentially self-adjoint, i.e.
		\[
		\overline{-\Delta}\;=\;(-\Delta)^{*}.
		\]
		Consequently there is a unique self-adjoint realisation (still denoted
		\(-\Delta\)) in \(L^{2}(M)\).
	\end{theorem}

	In the noncompact case the spectrum need not be discrete; one always has $\sigma(-\Delta)\subset[0,\infty).$
	Concrete descriptions depend on the geometry:
	\[
	\begin{array}{rclcl}
		M=\mathbb R^{n} &\Longrightarrow&
		\sigma(-\Delta)=[0,\infty)
		&\text{(purely continuous)}\\[4pt]
		M=\mathbb H^{n} &\Longrightarrow&
		\sigma(-\Delta)=[(n-1)^{2}/4,\;\infty)
		&\text{(purely continuous).}
	\end{array}
	\]
	
	Let \(E(\lambda)\) denote the projection–valued spectral measure of the unique self–adjoint extension of \(-\Delta\).  For each \(f\in L^2(M)\) the scalar measure
	\[
	\mu_f(B)=\langle E(B)f,f\rangle_{L^2(M)},\quad B\subset[0,\infty)\ \text{Borel}.
	\]
	With the projection–valued measure \(E(\lambda)\) of the positive operator
	\(-\Delta\), we define
	\[
	(-\Delta)_{spec}^{\,s}f
	:=\int_{0}^{\infty}\lambda^{\,s}\,dE(\lambda)\,f,
	\quad f\in\operatorname{Dom}\bigl((-\Delta)^s\bigr),
	\]
	where
	\[
	H^{2s}(M):=\operatorname{Dom}\bigl((-\Delta)_{spec}^{s}\bigr)
	=\Bigl\{f\in L^{2}(M):\int_0^\infty\lambda^{\,2s}\,dE_{f,f}(\lambda)<\infty\Bigr\}\]
	and
	\[\|f\|_{H^{2s}(M)}^2
=\int_{[0,\infty)}(1+\lambda)^{2s}\,dE_{f,f}(\lambda)
\sim\|f\|_2^2+\int_{[0,\infty)}\lambda^{2s}\,dE_{f,f}(\lambda).
	\]
	By Proposition \ref{fslps}, the following time integral is understood as a strong improper integral:
	\begin{equation}\label{fracspec}
		(-\Delta)_{B}^{\,s}f=(-\Delta)_{spec}^{\,s}f
		=\frac{s}{\Gamma(1-s)}
		\int_{0}^{\infty}\frac{f-e^{\,t\Delta}f}{t^{1+s}}\,dt
		\quad\text{in }L^{2}(M)\quad\;f\in H^{2s}(M).
	\end{equation}
	
	Recall that \((M,g)\) is stochastically complete when
	\[
	\int_{M}p_t(x,y)d\mathrm{vol}(y)=1
	\quad\text{for all }x\in M\text{ and }t>0.
	\]
	Set
	\[
	H(t,x) := e^{t\Delta}1(x) = \int_M p_t(x,y)\,d\mathrm{vol}(y)\in [0,1],\quad t > 0,
	\]
	and define
	\[
	H(0,x) := \lim_{t \to 0^+} H(t,x) = 1.
	\]
	Since \( e^{t\Delta} \) is a positivity-preserving semigroup, we have
	\[
	H(t+s,x) = e^{t\Delta} \bigl(e^{s\Delta}1\bigr)(x) \le H(t,x).
	\]
	Therefore, the limit \( \lim_{t \to \infty} H(t,x) \) exists and defines a function
	\begin{equation}\label{taoyi}
		H(x) := \lim_{t \to \infty} H(t,x) \in [0,1].
	\end{equation}
	The function \(H(x)\) quantifies the potential “loss of mass” due to diffusion escape in the general (possibly stochastically incomplete) case. Equivalently, if \(\tau\) denotes the explosion (or lifetime) of Brownian motion starting at \(x\), then
	\[
	H(x) = \mathbb{P}_x\{\tau = \infty\},
	\]
	that is, the probability that the diffusion does not explode. Clearly,
	\[
	0 \le H(x) \le 1,
	\]
	and \(M\) is said to be stochastically complete if and only if \(H(x) \equiv 1\) for all \(x \in M\).

	\begin{proposition}
		\(H\) is a nonnegative bounded harmonic function.
	\end{proposition}
	
	\begin{proof}
		By heat-kernel regularity (see \cite{grigoryan2009heat}), $H(t,x)$ is $C^{\infty}$ in $(0,\infty)\times M,$ and satisfies the heat equation
		\[
		\partial_tH(t,x)=\Delta_xH(t,x),
		\quad H(0,x)=1.
		\]
		
		For any $\varphi\in C_c^\infty(M)$, by the dominated convergence theorem
		\[
		\int_M H(x)\,\Delta\varphi(x)\,d\mathrm{vol}(x)
		=\lim_{t\to\infty}\int_M H(t,x)\,\Delta\varphi(x)\,d\mathrm{vol}(x).
		\]
		Since $H$ satisfies the heat equation,
		\[
		\int_M H(t,x)\,\Delta\varphi(x)\,d\mathrm{vol}(x)
		=\int_M (\partial_tH(t,x))\,\varphi(x)\,d\mathrm{vol}(x)
		=\frac{d}{dt}\int_M H(t,x)\,\varphi(x)\,d\mathrm{vol}(x).
		\]
		Set $F(t)=\int_M H(t,x)\varphi(x)\,d\mathrm{vol}(x)$.
The preceding identities show that $F'(t)$ has the finite limit
$\int_M H\Delta\varphi$. Since $F$ is bounded, that limit must be zero:
a nonzero real or imaginary part would force linear growth of $F$.
		Hence,
		\[
		\int_M H(x)\,\Delta\varphi(x)\,d\mathrm{vol}(x) = 0,
		\]
		i.e.\ $H$ is a distributional solution of $\Delta\: H=0$. Standard elliptic regularity implies $H\in C^\infty(M)$ and hence
		\[
		\Delta \:H(x)=0
		\quad\text{for all }x\in M.
		\]
		Thus, we obtain the desired results.
	\end{proof}
	
	On a stochastically complete manifold, the fractional formula is
\[
(-\Delta)^s_{\mathrm{spec}}f(x)
=\frac{s}{\Gamma(1-s)}\int_M[f(x)-f(y)]\mathcal K_s(x,y)\,d\mathrm{vol}(y),
\quad \mathcal K_s(x,y)=\int_0^\infty p_t(x,y)\,\frac{dt}{t^{1+s}},
\]
provided $f\in H^{2s}(M)$ and the corresponding double integral of
the absolute difference is finite. The right side then represents the
spectral operator almost everywhere. Boundedness of $f$ alone is
insufficient near $y=x$; other function classes may require a principal
value interpretation. On closed manifolds related integral formulas are
given in \cite{caselli2024fractional}.

	On a general (possibly non‐stochastically complete) manifold \((M,g)\), we require that the kernel
	\[
	\mathcal{K}_s(x,y)
	:=\int_{0}^{\infty}\frac{p_t(x,y)}{t^{1+s}}\,dt
	\]
	be well‐defined (i.e.\ the integral converges for all \(x\neq y\)), and we assume that \(f\) satisfies
	\[
	\int_{0}^{\infty}\int_{M}\bigl|f(x)-f(y)\bigr|\frac{p_t(x,y)}{t^{1+s}}
	\,d\mathrm{vol}(y)\,dt<\infty.
	\]
	Under these hypotheses, one may interchange the order of integration and define the heat‐kernel fractional Laplacian by
	\begin{equation}\label{rhfens}
		(-\Delta)_{\mathrm{hk}}^s f(x)
		:= \frac{s}{\Gamma(1-s)}
		\int_M\bigl(f(x)-f(y)\bigr)\,\mathcal{K}_s(x,y)\,d\mathrm{vol}(y).
	\end{equation}

	Then a natural question arises: on a non-stochastically complete manifold, does \((-\Delta)^s_{hk} f(x)\) coincide with the spectrally defined fractional Laplacian \((-\Delta)^s_{spec} f(x)\)? If not, what exactly is the difference between the two definitions?
	
	Let
	\[
	r(t,x)
	=1-\int_M p_t(x,y)\,d\mathrm{vol}(y).
	\]
	For $f\in H^{2s}(M)$ satisfying the preceding absolute pointwise conditions, the time-integral representative gives
	\begin{equation}\label{frachjhhjhjhjhjjjjghjhjgghjhjbgfjhjgbgbggggbhjgbhjjhjgbggbgbgbgbgbgggggbgbgbgbgbgbggbgbbgbgbbgbgbggbgbbggbgbgbgbgbfhgjfbhgjjjgbfhhhspec}
		(-\Delta)^s_{spec}f(x)
		=\frac{s}{\Gamma(1-s)}
		\int_0^\infty\frac{f(x)-e^{t\Delta}f(x)}{t^{1+s}}\,dt,
	\end{equation}
	and
	\[
	(-\Delta)^s_{\rm hk}f(x)
	=\frac{s}{\Gamma(1-s)}
	\int_0^\infty\!\int_M\frac{f(x)-f(y)}{t^{1+s}}\,p_t(x,y)\,d\mathrm{vol}(y)\,dt.
	\]
	Subtracting yields
	\[
	\bigl((-\Delta)^s_{spec}-(-\Delta)^s_{\rm hk}\bigr)f(x)
	=\frac{s}{\Gamma(1-s)}
	\int_0^\infty\frac{r(t,x)}{t^{1+s}}\,dt\;f(x),
	\]
	so that
	\[
	(-\Delta)^s_{\mathrm{spec}}f(x)
=V_s(x)f(x)+(-\Delta)^s_{\mathrm{hk}}f(x),
	\quad
	V_s(x)
	:=\frac{s}{\Gamma(1-s)}
	\int_0^\infty t^{-1-s}\,r(t,x)\,dt.
	\]
	The limit of $V_s$ will allow us to compare these expressions in $L^2_{\mathrm{loc}}$.

Next, we give the proof of Proposition \ref{chaz11}.

		\begin{proof}[\textbf{Proof of
			Proposition~\ref{chaz11}.}]
Fix a compact neighborhood $K$ of $x$ and choose
$0\le\chi\le1$, $\chi\in C_c^\infty(M)$, with $\chi=1$ on $K$.
The sub-Markov property and the semigroup formula give
\[
0\le r(t,z)\le1-P_t\chi(z)\le t\|\Delta\chi\|_\infty,
\qquad z\in K.
\]
In particular, $V_s(z)$ is finite for $0<s<1$.
For fixed $T\ge1$ the contribution of $(0,T)$ is bounded by
$C_KsT^{1-s}/[\Gamma(1-s)(1-s)]$, which tends to zero.
Given $\eta>0$, choose $T$ so that
$|H(t,x)-H(x)|<\eta$ for $t\ge T$. Then
\[
\left|\frac{s}{\Gamma(1-s)}\int_T^\infty
[H(t,x)-H(x)]t^{-1-s}\,dt\right|
\le\frac{T^{-s}}{\Gamma(1-s)}\eta,
\]
whereas the constant part tends to $1-H(x)$. This proves
$V_s(x)\to1-H(x)$.
The cutoff estimate and $r\le1$ also give a bound for $V_s$ uniform
on $K$ for $0<s\le s_0<1$.
\end{proof}

For $f\in H^\varepsilon(M)$, whenever the heat-kernel integrals above
are defined for all sufficiently small $s$, Proposition~\ref{slx11},
the local uniform bound for $V_s$, and dominated convergence imply
\[
(-\Delta)^s_{\mathrm{hk}}f\longrightarrow Hf-P_0f
\quad\text{in }L^2_{\mathrm{loc}}(M).
\]
No pointwise convergence follows solely from this spectral argument. If \(M\) is a stochastically complete Riemannian manifold, then \(H \equiv 1\) and
	\[
	\lim_{s \to 0^+} V_s(x) = 0 \quad \text{for all } x \in M,
	\]
	this is consistent with the conclusion of Proposition~\ref{chaz11}.

	Consequently, in the stochastically complete case, for
$f\in H^\varepsilon(M)$ on the common domain,
\[
(-\Delta)^s_{\mathrm{hk}}f=(-\Delta)^s_{\mathrm{spec}}f
\longrightarrow f-P_0f\quad\text{in }L^2(M).
\]
	We now turn to the definition and representation of the logarithmic Laplacian on non-compact, complete and connected Riemannian manifolds. 
	
	Define the logarithmic Sobolev space by
	\[
	H^{\log}(M)
	:= \left\{\, f \in L^2(M)\; \middle|\;
	\int_{(0,\infty)} (\log \lambda)^2\, dE_{f,f}(\lambda) < \infty \right\}
	\]
	equipped with the norm
	\[
	\|f\|_{H^{\log}}^2
	:= \|f\|_{L^2(M)}^2 + \int_{(0,\infty)} (\log \lambda)^2\, dE_{f,f}(\lambda).
	\]
	
	The logarithmic Laplacian is then defined spectrally by
	\begin{equation}\label{logspecs}
		\log(-\Delta)_{\mathrm{spec}}\,f
		:= \int_{(0,\infty)} \log \lambda\, dE(\lambda)\, f, \quad f \in H^{\log}(M).
	\end{equation}
	Moreover, Theorem~\ref{logboch11} gives
\[
\log(-\Delta)_{\mathrm{spec}}f
=\int_0^\infty\frac{e^{-t}Qf-P_tQf}{t}\,dt
\quad\text{in }L^2(M),\qquad P_t=e^{t\Delta}.
\]
For $P_0f=0$, pointwise evaluation additionally requires control of the
time integral, for example
\begin{equation}\label{shdja}
\int_0^\infty |e^{-t}f(x)-P_tf(x)|\,\frac{dt}{t}<\infty.
\end{equation}
Theorem~\ref{thm:log-noncompact11} supplies sufficient geometric and
heat-kernel hypotheses for compactly supported H\"older functions.
We give a proof that separates the short-time estimate from the
large-time assumption.

	\begin{proof}[\textbf{Proof of
		Theorem~\ref{thm:log-noncompact11}.}]
	By the Ricci lower bound, $M$ is stochastically complete
	\cite{yau1976}; hence $P_t1=1$. Fix a compact set $K\subset M$. Choosing normal balls with radii
	uniformly below the injectivity radius on $K$, we obtain
	$g_x\in C^\infty(M)$, $x\in K$, such that $	0\le g_x\le1,\, g_x(x)=0,$ and
	\[
	g_x(y)\ge c_K\min\{d(x,y)^2,1\},
	\qquad
	\sup_{x\in K}\|\Delta g_x\|_\infty\le C_K.
	\]
	Moreover, $1-g_x\in C_c^\infty(M)$ and
	\[
	P_tg_x(x)
	=
	(1-g_x)(x)-P_t(1-g_x)(x)
	=
	\int_0^tP_s\Delta(1-g_x)(x)\,ds
	\le C_Kt.
	\]
	Therefore,
	\[
	\int_Mp_t(x,y)\min\{d(x,y)^2,1\}\,d\mathrm{vol}(y)
	\le C_Kt.
	\]
	For $0<a\le2$, Jensen's inequality gives
	\begin{equation}\label{liyau}
		\int_Mp_t(x,y)\min\{d(x,y)^a,1\}\,d\mathrm{vol}(y)
		\le C_{K,a}t^{a/2},
		\qquad x\in K,\quad0<t\le1.
	\end{equation}
	
	Take $0<\beta\le\min\{\alpha,1\}$. Since $f\in C_c^\alpha(M)$, $|f(x)-f(y)|
	\le C\min\{d(x,y)^\beta,1\}.$
	Hence, by \eqref{liyau},
	\[
	\int_0^1\int_M
	p_t(x,y)|f(x)-f(y)|\,d\mathrm{vol}(y)\,\frac{dt}{t}<\infty
	\]
	locally uniformly in $x$.
	
	Let $S\subset M$ be compact with $\supp f\subset S$. By Jensen,
	symmetry, and \eqref{liyau} with $a=2\beta$,
	\begin{align}
		\|f-P_tf\|_2^2
		&\le
		\iint_{M\times M}
		p_t(x,y)|f(x)-f(y)|^2\,d\mathrm{vol}(x)d\mathrm{vol}(y) \notag\\
		&\le
		2\int_S\int_M
		p_t(x,y)|f(x)-f(y)|^2\,d\mathrm{vol}(y)d\mathrm{vol}(x)
		\le Ct^\beta.
		\label{yjsh}
	\end{align}
	Thus
	\[
	\|e^{-t}f-P_tf\|_2
	\le
	(1-e^{-t})\|f\|_2+\|f-P_tf\|_2
	\le C(t+t^{\beta/2}),
	\]
	and consequently
	\[
	\int_0^1
	\|e^{-t}f-P_tf\|_2\,\frac{dt}{t}<\infty.
	\]
	
	For $t>0$, the semigroup property and Cauchy--Schwarz imply
	\[
	p_t(x,y)\le
	\sqrt{p_t(x,x)p_t(y,y)}.
	\]
	Hence, if $A\subset M$ is compact and
	$K\supset A\cup S$ is compact, $	P_t|f|(x)\le q_K(t)\|f\|_1,\,x\in A$,
	while
	\[
	\|P_tf\|_2^2
	=
	\langle f,P_{2t}f\rangle
	\le q_S(2t)\|f\|_1^2.
	\]
	Since $t\mapsto p_t(x,x)$ is nonincreasing,
	\[
	q_K(t)\le q_K(1)^{1/2}q_K(t)^{1/2},
	\qquad t\ge1.
	\]
	Thus \eqref{heat-tail} yields
	\[
	\int_1^\infty q_K(t)\,\frac{dt}{t}<\infty,
	\qquad
	\int_1^\infty\|P_tf\|_2\,\frac{dt}{t}<\infty.
	\]
	
	Moreover, \eqref{heat-tail} implies $q_K(t)\to0$ as $t\to\infty$.
	For $u\in C_c^\infty(M)$,
	\[
	\|P_tu\|_2^2
	\le
	q_{\supp u}(2t)\|u\|_1^2\longrightarrow0.
	\]
	On the other hand, $P_tu\rightarrow E(\{0\})u$ in $L^2(M),$
	so $E(\{0\})u=0$. By density, $	E(\{0\})=0.$
	
	Define
	\[
	T_{\varepsilon,R}f
	:=
	\int_\varepsilon^R
	(e^{-t}f-P_tf)\,\frac{dt}{t}
	=
	m_{\varepsilon,R}(\Delta)f,
	\]
	where
	\[
	m_{\varepsilon,R}(\lambda)
	=
	\int_\varepsilon^R
	\frac{e^{-t}-e^{-t\lambda}}{t}\,dt.
	\]
	The preceding estimates imply that
	$T_{\varepsilon,R}f$ converges in $L^2(M)$ as
	$\varepsilon\downarrow0$ and $R\uparrow\infty$. Hence
	\[
	\sup_{\varepsilon,R}
	\int_{(0,\infty)}
	|m_{\varepsilon,R}(\lambda)|^2\,dE_{f,f}(\lambda)<\infty.
	\]
	Since $	m_{\varepsilon,R}(\lambda)\rightarrow\log\lambda,\, \lambda>0,$
	Fatou's lemma gives $	\int_{(0,\infty)}
	|\log\lambda|^2\,dE_{f,f}(\lambda)<\infty.$
	Thus $f\in H^{\log}(M)$, and Theorem~\ref{logboch11} gives
	\[
	\log(-\Delta)_{\mathrm{spec}}f
	=
	\int_0^\infty
	(e^{-t}f-P_tf)\,\frac{dt}{t}
	\qquad\text{in }L^2(M).
	\]
	
	For $0<t<1$, stochastic completeness gives
	\[
	e^{-t}f(x)-P_tf(x)
	=
	(e^{-t}-1)f(x)
	+
	\int_Mp_t(x,y)[f(x)-f(y)]\,d\mathrm{vol}(y).
	\]
	Therefore, by Fubini,
	\begin{align*}
		\log(-\Delta)_{\mathrm{spec}}f(x)
		&=
		\int_M
		\left(\int_0^1p_t(x,y)\,\frac{dt}{t}\right)
		[f(x)-f(y)]\,d\mathrm{vol}(y)\\
		&\quad
		-\int_M
		\left(\int_1^\infty p_t(x,y)\,\frac{dt}{t}\right)
		f(y)\,d\mathrm{vol}(y)\\
		&\quad+
		\left(
		\int_0^1\frac{e^{-t}-1}{t}\,dt
		+
		\int_1^\infty\frac{e^{-t}}{t}\,dt
		\right)f(x).
	\end{align*}
	By Lemma~\ref{euler},
	\[
	\int_0^1\frac{e^{-t}-1}{t}\,dt
	+
	\int_1^\infty\frac{e^{-t}}{t}\,dt
	=
	\Gamma'(1).
	\]
	Hence
	\[
	\log(-\Delta)_{\mathrm{spec}}f(x)
	=
	\int_M K_1(x,y)[f(x)-f(y)]\,d\mathrm{vol}(y)
	-
	\int_M K_2(x,y)f(y)\,d\mathrm{vol}(y)
	+
	\Gamma'(1)f(x).
	\]
	
	The previous estimates are locally uniform in $x$. Since every
	truncated time integral is continuous, the above expression is the
	locally uniform limit of continuous functions and therefore defines a
	continuous representative of
	$\log(-\Delta)_{\mathrm{spec}}f$.
\end{proof}

	\section{Logarithmic Laplacian on Hyperbolic Space \texorpdfstring{\(\mathbb{H}^n\)}{hyperbolic n-space}}
	
	In this section, we derive explicit formulas for the fractional Laplacian and the logarithmic Laplacian on real hyperbolic space, which is complete and stochastically complete. We then explore its spectral and heat kernel representations and discuss further analytic and geometric properties.
	
	\subsection{Geometry and heat kernel of  \texorpdfstring{\(\mathbb{H}^n\)}{hyperbolic n-space}}
	
	In this subsection, we recall some basic geometric properties of hyperbolic space \(\mathbb{H}^n\) and summarize key estimates for its heat kernel. It is important to note that the essential features of hyperbolic geometry only emerge in dimensions \(n \ge 2\). Therefore, in the following, we restrict our attention to the case \(n \ge 2\).
	
	The standard model for real hyperbolic \(n\)-space is the Poincaré disc:
	\[
	\mathbb{H}^n
	=\bigl\{\,x\in \mathbb{R}^n:\lvert x\rvert<1\,\bigr\},
	\]
	equipped with the metric
	\[
	g
	=\frac{4\,|dx|^2}{\bigl(1-\lvert x\rvert^2\bigr)^2}.
	\]
	Thus the Riemannian volume element is
	\[
	d\mathrm{vol}(x)
	=2^{n}\bigl(1-\lvert x\rvert^2\bigr)^{-n}\,dx,
	\]
	where \(dx\) is the Euclidean Lebesgue measure on \(\{\,\lvert x\rvert<1\,\}\).
	
	The geodesic distance from the origin is
	\[
	\rho = d_{\mathbb{H}^n}(0,x)
	=2\,\tanh^{-1}(r)
	= \log\!\bigl(\tfrac{1+r}{1-r}\bigr).
	\]
	Equivalently, \(r = \tanh(\tfrac{\rho}{2})\).
	
	\medskip
	In these geodesic polar coordinates, the metric becomes
	\[
	g
	= d\rho^2 \;+\; \sinh^2(\rho)\,d\omega^2,
	\]
	where \(d\omega^2\) denotes the standard round metric on \(\mathbb{S}^{n-1}\). The associated volume element is
	\[
	d\mathrm{vol}_{\mathbb{H}^n} = \sinh^{n-1}(\rho)\,d\rho\,d\omega,
	\]
	thus, the volume of a geodesic ball \(B(\rho)\subset \mathbb{H}^n\) of radius \(\rho>0\) is given by
	\[
	\mathrm{Vol}(B(\rho)) = \mathrm{Vol}(\mathbb{S}^{n-1})\int_0^\rho \sinh^{n-1}(t)\,dt,
	\]
	where \(\mathrm{Vol}(\mathbb{S}^{n-1})\) is the surface area of the unit sphere in \(\mathbb{R}^n\).
	
	Let \(p_n(r,t)\) denote the heat kernel on the real hyperbolic space \(\mathbb{H}^n\), where $r \;=\; d_{\mathbb{H}^n}(x,y)$ is the geodesic distance between points \(x\) and \(y\), and \(t>0\) is the time variable.  By definition, \(p_t(x,y)\) is the minimal fundamental solution of the heat equation
	\[
	\bigl(\partial_{t} - \Delta_{\mathbb{H}^n}\bigr)\,u(x,t) \;=\; 0,
	\]
	subject to the initial condition 
	\[
	\lim_{t\to0^+} u(x,t) \;=\; \delta_{y}(x).
	\]
	Equivalently, \(p_{n}(r,t)\) is the kernel of the heat semigroup \(e^{t\Delta}\).  Consequently, for any \(f\in L^{2}(\mathbb{H}^n)\), the function
	\begin{equation}\label{heat equa}
		u(x,t) :=e^{t\Delta_{\mathbb{H}^n}}f(x)\;=\; \int_{\mathbb{H}^n} p_{n}\bigl(r,\,t\bigr)\;f(y)\,d\mathrm{vol}(y)
	\end{equation}
	solves 
	\[
	\bigl(\partial_{t} - \Delta_{\mathbb{H}^n}\bigr)\,u(x,t) \;=\; 0,
	\qquad
	u(x,0) \;=\; f(x).
	\]
	
	\begin{theorem}\cite[Theorem 1.1]{grigor1998heat}\label{heatfo}
		The heat kernel $p_n(r,t)$ on the hyperbolic space \(\mathbb{H}^n\) is given by the following formulas:
		
		If $n=2m+1,m\ge0,$ then
		\begin{equation}\label{sqkjp}
			p_{n}(r,t)
			=\frac{(-1)^{m}}{2^{m}\pi^{m}}
			\frac{1}{(4\pi t)^{1/2}}
			\Bigl(\frac{1}{\sinh r}\,\partial_{r}\Bigr)^{\!m}
			\exp\!\Bigl(-m^{2}t-\frac{r^{2}}{4t}\Bigr);
		\end{equation}
		if $n=2m+2,m\ge 0,$ then
		\[
		p_{n}(r,t)
		=\frac{(-1)^{m}}{2^{m+5/2}\pi^{m+3/2}}\,
		t^{-3/2}\exp\!\Bigl(-\tfrac{(2m+1)^{2}}{4}t\Bigr)
		\Bigl(\frac{1}{\sinh r}\,\partial_{r}\Bigr)^{\!m}
		\int_{r}^{\infty}
		\frac{x\,\exp\!\left(-\tfrac{x^{2}}{4t}\right)}
		{\bigl(\cosh x-\cosh r\bigr)^{1/2}}
		\,\mathrm{d}x .
		\]
	\end{theorem}
	
	In \cite{davies1988heat}, Davies and Mandouvalos obtain the equivalent asymptotic property of the heat kernel, see Proposition \ref{prop:heat-asymp}.

	\begin{proposition}\cite[Theorem 3.1]{davies1988heat}\label{prop:heat-asymp}
		Let \(p_{n}(r,t)\) be the heat kernel on \(\mathbb{H}^{n}\), where \(r=d_{\mathbb{H}^n}(x,y)\) and \(n\ge 2\).  Then for each fixed \(r\ge0\) and all \(t>0\)
		\[
		p_{n}(r,t) \sim
		t^{-\,\frac{n}{2}} \,
		\exp\!\left\{-\,\frac{(n-1)^{2}}{4}\,t \;-\;\frac{r^{2}}{4t}\;-\;\frac{(n-1)\,r}{2}\right\}
		\;\cdot\;(1+r+t)^{\frac{n-3}{2}}\,(1+r)
		\]
		uniformly in \(0\le r<\infty\) and \(0<t<\infty\).
	\end{proposition}

	\subsection{Pointwise representation of fractional Laplacian}
	
	On $\mathbb H^n$ the heat semigroup preserves mass. The spectral
and heat-kernel formulas therefore agree on their common domain, as in
\eqref{rhfens}. Under the stated absolute integrability condition one has
	\[
	(-\Delta_{\mathbb{H}^n})_{\mathrm{hk}}^{\,s}f(x)
	= \frac{s}{\Gamma(1-s)}
	\int_{\mathbb{H}^n} \bigl(f(x) - f(y)\bigr)\, \mathcal{K}_s(r)d\mathrm{vol}(y),
	\]
	where the kernel \(\mathcal{K}_s(r)\) is positive and given by
	\begin{equation}\label{fenshureh}
		\mathcal{K}_s(r) := \int_0^\infty p_n(r,t)\, \frac{dt}{t^{1+s}},
	\end{equation}
	where \(r = d_{\mathbb{H}^n}(x,y)\) denotes the hyperbolic distance between \(x\) and \(y\).
	
	For $f\in C_c^\infty(\mathbb H^n)$ and $0<s<1$, the same formula
holds with a principal value at $y=x$. In geodesic normal coordinates,
the spherical average of the linear Taylor term vanishes because the
kernel and volume density are radial; the remaining difference is
$O(r^2)$ and is integrable against $r^{-n-2s}d\mathrm{vol}$.
The heat-kernel bounds below justify passage to this principal value.
For merely bounded $f$, a pointwise formula requires further regularity.

Next, we calculate $\mathcal K_s(r)$ in odd dimension \(n=2m+1,m>0\).
	
	We start from
	\[
	\mathcal{K}_{s}(r)\;=\;\int_0^\infty p_{2m+1}(r,t)\,\frac{dt}{t^{1+s}},r>0
	\]
	with
	\[
	p_{2m+1}(r,t)
	=\frac{(-1)^m}{2^m\pi^m}\,
	\frac{1}{(4\pi t)^{1/2}}
	\Bigl(\tfrac1{\sinh r}\partial_r\Bigr)^m
	\exp\!\Bigl(-m^2t - \tfrac{r^2}{4t}\Bigr).
	\]
	Hence
	\[
	\mathcal{K}_{s}(r)
	=\frac{(-1)^m}{2^m\pi^m\,(4\pi)^{1/2}}
	\Bigl(\tfrac1{\sinh r}\partial_r\Bigr)^m
	\int_0^\infty
	e^{-m^2t - \tfrac{r^2}{4t}}\,
	t^{-\tfrac32 - s}\,dt.
	\]
	For $z>0$ and $\nu\in\mathbb R$, the modified Bessel function of the second kind admits the Laplace‐type integral representation \cite[Eq.~10.32.10]{olver2010nist}
	\[
	K_{\nu}(z)
	=\frac12\Bigl(\frac{z}{2}\Bigr)^{\!\nu}
	\int_{0}^{\infty}
	\exp\!\Bigl(-t-\frac{z^{2}}{4t}\Bigr)\,t^{-\nu-1}\,\mathrm{d}t.
	\]
	Thus,
	\begin{equation}\label{bessel2}
		\int_0^\infty
		t^{-\nu-1}
		\exp\!\Bigl(-\alpha t - \tfrac\beta t\Bigr)\,dt
		=2\bigl(\tfrac\beta\alpha\bigr)^{-\nu/2}\!
		\,K_{\nu}\!\bigl(2\sqrt{\alpha\beta}\bigr),
		\quad \alpha,\beta>0.
	\end{equation}
	Here set
	\(\alpha=m^2,\ \beta=\tfrac{r^2}{4},\ \nu=s+\tfrac12\).  Then
	\[
	\int_0^\infty
	e^{-m^2t - \tfrac{r^2}{4t}}t^{-\tfrac32 - s}\,dt
	=2\bigl(\tfrac{r^2}{4m^2}\bigr)^{-\tfrac{2s+1}{4}}
	K_{\,s+\tfrac12}(m r).
	\]
	Since \(\bigl(\tfrac{r^2}{4m^2}\bigr)^{-\tfrac{2s+1}{4}}
	=2^{\,s+\tfrac12}\,m^{\,s+\tfrac12}\,r^{-\,s-\tfrac12}\),
	we obtain that for odd $n\ge 3$
	\[
	\begin{aligned}
		\mathcal{K}_{s}(r)
		=&\frac{(-1)^m}{2^m\pi^m\,(4\pi)^{1/2}}
		\Bigl(\tfrac1{\sinh r}\partial_r\Bigr)^m
		\Bigl[\,2^{\,s+\tfrac32}\,m^{\,s+\tfrac12}\,r^{-\,s-\tfrac12}\,
		K_{\,s+\tfrac12}(m r)\Bigr]\\=&\left(-1\right)^{m}\frac{m^{s+\frac{1}{2}}}{2^{\,m-s-\frac{1}{2}}\,\pi^{m+\frac{1}{2}}}
		\Bigl(\tfrac1{\sinh r}\partial_r\Bigr)^m
		\!\Bigl[r^{-s-\frac{1}{2}}\,K_{\,s+\tfrac12}(m r)\Bigr]\\=&\left(-1\right)^{\frac{n-1}{2}}\frac{(n-1)^{s+\frac{1}{2}}}{2^{\frac{n-1}{2}}\,\pi^{\frac{n}{2}}}
		\Bigl(\tfrac1{\sinh r}\partial_r\Bigr)^{\frac{n-1}{2}}
		\!\Bigl[r^{-s-\frac{1}{2}}\,K_{\,s+\tfrac12}(\frac{n-1}{2} r)\Bigr].
	\end{aligned}
	\]
	
	For even dimension \(n=2m+2,m\ge0\), similarly, we begin with the definition
	\[
	\mathcal{K}_{s}(r)
	=\int_0^\infty p_{2m+2}(r,t)\,\frac{dt}{t^{1+s}},
	\]
	and the known heat‐kernel formula
	\[
	p_{2m+2}(r,t)
	=\frac{(-1)^m}{2^{m+\tfrac52}\,\pi^{m+\tfrac32}}\,
	t^{-\tfrac32}\,e^{-\tfrac{(2m+1)^2}{4}\,t}\,
	\Bigl(\tfrac1{\sinh r}\partial_r\Bigr)^{\!m}
	\!\int_r^\infty\frac{x\,e^{-\,x^2/(4t)}}{\sqrt{\cosh x-\cosh r}}\,dx.
	\]
	Thus
	\[
	\begin{aligned}
		\mathcal{K}_{s}(r)
		&=\frac{(-1)^m}{2^{m+\tfrac52}\,\pi^{m+\tfrac32}}
		\Bigl(\tfrac1{\sinh r}\partial_r\Bigr)^{\!m}
		\int_r^\infty\frac{x}{\sqrt{\cosh x-\cosh r}}
		\Bigl[\int_0^\infty t^{-s-\frac{5}{2}}\,
		e^{-\,\frac{(2m+1)^2}{4}t - \tfrac{x^2}{4t}}
		\,dt\Bigr]
		\,dx.
	\end{aligned}
	\]
	Use formula (\ref{bessel2}) again
	with \(\nu=s+\frac{3}{2}\), \(\alpha=(2m+1)^2/4\), and \(\beta=x^2/4\).  Then
	\[
	\int_0^\infty t^{-s-\frac{5}{2}}\,
	e^{-\frac{(2m+1)^2}{4}t - \frac{x^2}{4t}}\,dt
	=2\Bigl(\tfrac{x^2}{(2m+1)^2}\Bigr)^{-\frac{2s+3}{4}}
	K_{\,s+\frac{3}{2}}\!\bigl((m+\tfrac12)x\bigr).
	\]
	We obtain that for even $n\ge 2$
	\[
	\begin{aligned}
		\mathcal{K}_{s}(r)
		=&\left(-1\right)^{m}\frac{(2m+1)^{s+\frac{3}{2}}}{2^{\,m+\frac32}\,\pi^{\,m+\frac32}}
		\Bigl(\tfrac1{\sinh r}\partial_r\Bigr)^{\!m}
		\int_r^\infty
		\frac{x^{-s-\frac{1}{2}}}{\sqrt{\cosh x-\cosh r}}\,
		K_{\,s+\frac{3}{2}}\!\bigl((m+\tfrac12)x\bigr)
		\,dx\\=&\left(-1\right)^{\frac{n}{2}-1}\frac{(n-1)^{s+\frac{3}{2}}}{2^{\,\frac{n+1}{2}}\,\pi^{\frac{n+1}{2}}}
		\Bigl(\tfrac1{\sinh r}\partial_r\Bigr)^{\frac{n}{2}-1}
		\int_r^\infty
		\frac{x^{-s-\frac{1}{2}}}{\sqrt{\cosh x-\cosh r}}\,
		K_{\,s+\frac{3}{2}}\!\bigl(\frac{n-1}{2}x\bigr)
		\,dx.
	\end{aligned}
	\]
	These formulas give the dimensional constants for the unnormalized
kernel $\mathcal K_s$ in \eqref{fenshureh}. The factor $s/\Gamma(1-s)$
belongs to the operator formula. They are consistent, after this
normalization, with the hyperbolic Fourier representation in
\cite{banica2015some}; see also the three-dimensional formulas in
\cite{kim2021harnack}. Differentiation under the time integrals is valid
locally for $r>0$: on compact subintervals of $(0,\infty)$ the Gaussian
factor controls every derivative at small time and the positive spectral
bottom controls large time. In the even-dimensional formula the endpoint
singularity can first be removed by the substitution
$\cosh x=\cosh r+v^2$.
	
	Moreover, we directly recover the same asymptotic behavior of $\mathcal{K}_s(r)$ as in \cite[Theorem 2.4]{banica2015some}.  To this end, we begin with the following lemma.

	\begin{lemma}[Tail estimate for an incomplete Gamma integral]\label{lem:gamma-tail}
		Let \(n\ge 2,\:s\ge 0\), set 
		\[
		I(n,r):=\int_{\,r^{2}/4}^{\infty}t^{\frac{n}{2}+s-1}\,e^{-t}\,dt,
		\quad r>0.
		\]
		Then for $r>\sqrt{2n+4s-4}$, we derive that
		\[
		2^{2-2s-n}\,r^{\,n+2s-2}\,
		e^{-\,\frac{r^{2}}{4}}\le I(n,r)\le\;
		2^{2-2s-n}\,r^{\,n+2s-2}\,
		e^{-\,\frac{r^{2}}{4}}
		\Bigl(1-\tfrac{2(n-2+2s)}{r^{2}}\Bigr)^{-1}.
		\]
		In particular, whenever \(r\ge2\sqrt{n-2+2s}\),
		\[
		I(n,r)\;\le2^{3-2s-n}\,r^{\,n+2s-2}\,
		e^{-\,\frac{r^{2}}{4}}.
		\]
	\end{lemma}
	
	\begin{proof}
		Write the integral in terms of the upper incomplete Gamma function:
		\[
		I(n,r)=\Gamma\!\Bigl(\tfrac{n}{2}+s,\,\tfrac{r^{2}}{4}\Bigr),\quad
		a:=\tfrac{n}{2}+s\ge 1,\;x:=\tfrac{r^{2}}{4},
		\]
		where 
		\[\Gamma\left(a,x\right)=\int_x^{\infty}t^{a-1}e^{-t}dt.\]
		
		Integrate once by parts:
		\[
		\Gamma(a,x)=x^{a-1}e^{-x}+(a-1)\Gamma(a-1,x).
		\]
		Note that
		\(\Gamma(a-1,x)\le x^{-1}\Gamma(a,x)\).  Hence
		\[
		\Gamma(a,x)\;\le\;x^{a-1}e^{-x}
		\Bigl(1-\tfrac{a-1}{x}\Bigr)^{-1}.
		\]
		Substituting \(a=\tfrac{n}{2}+s\) and \(x=\tfrac{r^{2}}{4}\) gives
		\begingroup\small	\[
		I(n,r)\;\le\;
		\Bigl(\tfrac{r^{2}}{4}\Bigr)^{\frac{n}{2}+s-1}
		e^{-\,\frac{r^{2}}{4}}
		\Bigl(1-\tfrac{2(n-2+2s)}{r^{2}}\Bigr)^{-1}
		=2^{2-2s-n}\,r^{\,n+2s-2}\,
		e^{-\,\frac{r^{2}}{4}}
		\Bigl(1-\tfrac{2(n-2+2s)}{r^{2}}\Bigr)^{-1}.
		\]
		\endgroup
		On the other hand, it is obvious that for $n\ge 2,$ 
		\begingroup\small	\[	
		I(n,r)=\int_{\,r^{2}/4}^{\infty}t^{\frac{n}{2}+s-1}\,e^{-t}\,dt\ge 2^{2-2s-n}\,r^{\,n+2s-2}\int_{\,r^{2}/4}^{\infty}e^{-t}\,dt=2^{2-2s-n}\,r^{\,n+2s-2}\,
		e^{-\,\frac{r^{2}}{4}},
		\]
		\endgroup
		which is the desired bound.
	\end{proof}
	
	We use a version of Laplace's method with a fixed phase and a
controlled parameter-dependent amplitude; compare \cite[Chapter 2]{murray1984asymptotic}.
The hypotheses needed here are included explicitly.

\begin{lemma}\label{lem:laplace}
Let $\phi:(0,\infty)\to\mathbb R$ be continuous with a unique minimum
at $u_0>0$. Suppose it is $C^2$ near $u_0$,
$\phi''(u_0)>0$, and for every sufficiently small $\delta>0$,
\[
\inf_{|u-u_0|\ge\delta}\phi(u)>\phi(u_0).
\]
Let $a_r\ge0$ be measurable, converging uniformly near $u_0$ to a
function $a$ continuous there, with $a(u_0)>0$.
Suppose also that $a_r(u)\le A(u)$ for all large $r$, where
$\int_0^\infty A(u)e^{-\phi(u)}\,du<\infty$.
Then
\[
\lim_{r\to\infty}
\frac{\displaystyle\int_0^\infty a_r(u)e^{-r\phi(u)}\,du}
{\displaystyle a(u_0)e^{-r\phi(u_0)}
\sqrt{2\pi/(r\phi''(u_0))}}=1.
\]
\end{lemma}
\begin{proof}
Fix a small neighborhood of $u_0$. On its complement the phase exceeds
$\phi(u_0)$ by some $\eta>0$, so the integral there is bounded by
\[
e^{-(r-1)(\phi(u_0)+\eta)}\int_0^\infty A(u)e^{-\phi(u)}\,du,
\]
which is negligible compared with $e^{-r\phi(u_0)}/\sqrt r$.
Near $u_0$, Taylor's formula gives
$\phi(u)-\phi(u_0)=\tfrac12\phi''(u_0)(u-u_0)^2+o((u-u_0)^2)$
and a positive quadratic lower bound. Substitute
$v=\sqrt r\,(u-u_0)$. Local uniform convergence of $a_r$ and the
quadratic bound permit dominated convergence against a Gaussian,
yielding $a(u_0)\sqrt{2\pi/\phi''(u_0)}$ for the rescaled integral.
\end{proof}

	\begin{lemma}\label{lem:K2-asymp}
		Let \(n\ge2,s\ge 0\) and define
		\[
		J(n,r)\;:=\;
		\int_{0}^{r^2/4}t^{\frac n2+s-1}\exp\!\left\{-\,\frac{(n-1)^{2}r^2}{16t} \;-t\right\}
		\;\cdot\;(1+r+\frac{r^2}{4t})^{\frac{n-3}{2}} dt,
		\quad r>0.
		\]
		Then there exist positive constants $c_1,c_2,c_3,c_4$  such that
		\[c_{1}r^{\,n+s-2}\,e^{-\frac{(n-1)r}{2}}\le J(n,r)\le c_2 r^{\,n+s-2}\,e^{-\frac{(n-1)r}{2}},\quad \text{when}\:\:\:r\ge 1;\]
		and
		\[c_{3}r^{n+2s}\le J(n,r)\le c_4 r^{n+2s},\quad \text{when}\:\:\:r\le 1.\]
	\end{lemma}

	\begin{proof}
		For large $r$, set $t=ru$ and put
\[
\phi(u)=\frac{(n-1)^2}{16u}+u,\qquad
a_r(u)=u^{n/2+s-1}
\left(1+\frac1r+\frac1{4u}\right)^{(n-3)/2}\mathbf1_{\{u<r/4\}}.
\]
Then the change of variables is exact:
\[
J(n,r)=r^{n+s-3/2}\int_0^\infty a_r(u)e^{-r\phi(u)}\,du.
\]
The phase has its unique minimum at $u_0=(n-1)/4$, with
$\phi(u_0)=(n-1)/2$ and $\phi''(u_0)=8/(n-1)>0$.
It tends to infinity at both endpoints. Near $u_0$, $a_r$ converges
uniformly to
$a(u)=u^{n/2+s-1}(1+1/(4u))^{(n-3)/2}$, which is positive there.
For $r\ge1$, a constant multiple of
\[
A(u)=u^{n/2+s-1}(1+u^{-1})^{\max\{(n-3)/2,0\}}
\]
dominates $a_r$, and $Ae^{-\phi}$ is integrable at both endpoints.
Lemma~\ref{lem:laplace} therefore gives
\[
J(n,r)\sim r^{n+s-2}e^{-(n-1)r/2}\qquad r\to\infty.
\]
Positive continuity of $J$ extends the two-sided estimate over any
remaining compact range of $r\ge1$.

		For \(r\ll1\), note that on \(0<t\le r^2/4\):
		\[
		\exp(-t)\sim 1, 
		\quad
		1+r+\frac{r^2}{4t}\sim \frac{r^2}{t}.
		\]
		Hence
		\[
		J(n,r)\sim
		\int_{0}^{r^2/4}
		t^{\frac n2+s-1}
		\exp\!\bigl(-\tfrac{(n-1)^2r^2}{16t}\bigr)
		\Bigl(\tfrac{r^2}{t}\Bigr)^{\frac{n-3}{2}}
		\,dt.
		\]
		This simplifies to
		\[
		J(n,r)\sim
		r^{\,n-3}
		\int_{0}^{r^2/4}
		t^{s+\frac12}
		\exp\!\Bigl(-\tfrac{(n-1)^2r^2}{16t}\Bigr)
		\,dt.
		\]
		Substituting \(u=(n-1)^2r^2/(16t)\), one finds
		\[
		\int_{0}^{r^2/4}
		t^{s+\frac12}
		e^{-\,\frac{(n-1)^2r^2}{16t}}
		\,dt
		\sim
		r^{2s+3}
		\int_{(n-1)^2/4}^{\infty}
		u^{-s-\frac52}e^{-u}\,du
		\;\sim\;
		r^{2s+3}.
		\]
		Therefore, \(J(n,r)\sim r^{\,n-3}\,r^{2s+3}=r^{\,n+2s}\), as claimed.
	\end{proof}

Following these lemmas, we give the proof of Proposition \ref{fensjx11}.

	\begin{proof}[\textbf{Proof of
		Proposition~\ref{fensjx11}.}]
		Recall that
		\[
		\mathcal K_s(r)=\int_0^\infty p_n(r,t)t^{-1-s}\,dt,
		\]
		and split the integral at \(t=1\): $\mathcal K_s(r)
		=I_1(r)+I_2(r)$ where
		\[
		I_1(r)=\int_0^1p_n(r,t)t^{-1-s}\,dt,
\qquad I_2(r)=\int_1^\infty p_n(r,t)t^{-1-s}\,dt.
		\]
		\noindent\textbf{1. Behavior as \(r\to0\).}
		
		By Proposition~\ref{prop:heat-asymp}, for \(0<t\le1\) and \(0\le r\le1\),
		\[
		p_n(r,t)\;\sim\;t^{-\frac n2}\exp\!\Bigl(-\tfrac{r^2}{4t}\Bigr),
		\]
		since other factors remain bounded above and below.  Hence
		\[
		I_1(r)\;\sim\;\int_{0}^{1}t^{-\frac n2-1-s}\exp\!\Bigl(-\tfrac{r^2}{4t}\Bigr)\,dt.
		\]
		Make the substitution \(u = r^2/(4t)\), so that \(dt = -(r^2/(4u^2))\,du\).  Then
		\[
		I_1(r)
		\;\sim\;
		r^{-n-2s}
		\int_{r^2/4}^{\infty}u^{\frac n2+s-1}e^{-u}\,du
		\sim r^{-n-2s}
		\quad r\to0.
		\]
		
		On the other hand, for \(t\ge1\) and \(r\le1\), Proposition~\ref{prop:heat-asymp} gives
		\[
		p_n(r,t)\;\sim\;t^{-\frac n2}\left(1+r+t\right)^{\frac{n-3}{2}}e^{-\frac{\left(n-1\right)^2}{4}t}\exp\!\Bigl(-\tfrac{r^2}{4t}\Bigr),
		\]
		so by Lemma \ref{lem:K2-asymp}
		\[
		I_2(r)\sim 
		r^{-n-2s}
		\int_{0}^{r^2/4}u^{\frac n2+s-1}\left(1+r+\frac{r^2}{4u}\right)^{\frac{n-3}{2}}e^{-u}e^{-\frac{\left(n-1\right)^{2}r^2}{16u}}\,du \sim 1
		\quad r\to0.
		\]
		Thus \(I_2(r)\) does not affect the leading asymptotics.  We conclude
		\[
		\mathcal K_s(r)\sim r^{-n-2s}
		\quad(r\to0).
		\]
		\noindent\textbf{2. Behavior as \(r\to\infty\).}
		
		For \(0<t\le1\) and \(r\gg1\), Proposition~\ref{prop:heat-asymp} yields
		\[
		p_n(r,t)\;\sim\;
		t^{-\frac n2}
		\exp\!\Bigl(-\tfrac{r^2}{4t}-\tfrac{(n-1)r}{2}\Bigr)\left(1+r\right)^{\frac{n-1}{2}}.
		\]
		Hence, by Lemma \ref{lem:gamma-tail}
		\[
		\begin{aligned}
			I_{1}(r)
			\sim&
			(1+r)^{\frac{n-1}{2}}
			e^{-\frac{(n-1)r}{2}}\;
			\int_{0}^{1} t^{-\frac{n}{2}-1-s}e^{-\frac{r^2}{4t}}\,dt\\ \sim& r^{-n-2s} (1+r)^{\frac{n-1}{2}}
			e^{-\frac{(n-1)r}{2}}\;
			\int_{r^2/4}^{\infty} u^{\frac{n}{2}+s-1}e^{-u}\,du \\ \sim& r^{-n-2s} (1+r)^{\frac{n-1}{2}}
			e^{-\frac{(n-1)r}{2}}\;
			r^{\,n+2s-2}\,
			e^{-\,\frac{r^{2}}{4}} \\=&r^{-2}(1+r)^{\frac{n-1}{2}}	e^{-\frac{r^2}{4}-\frac{(n-1)r}{2}}.
		\end{aligned}
		\]
		
		For \(t\ge1\), still with \(r\gg1\), one has
		\[
		p_n(r,t)\;\sim\;
		rt^{-\,\frac{n}{2}} \,
		\exp\!\left\{-\,\frac{(n-1)^{2}}{4}\,t \;-\;\frac{r^{2}}{4t}\;-\;\frac{(n-1)\,r}{2}\right\}
		\;\cdot\;(1+r+t)^{\frac{n-3}{2}},
		\]
		so by Lemma \ref{lem:K2-asymp} we obtain that
		\[
		\begin{aligned}
			I_2(r)
			\sim&
			re^{-\frac{(n-1)r}{2}}
			\int_{1}^{\infty}t^{-\frac n2-1-s}\exp\!\left\{-\,\frac{(n-1)^{2}}{4}\,t \;-\;\frac{r^{2}}{4t}\right\}
			\;\cdot\;(1+r+t)^{\frac{n-3}{2}} dt\\\sim
			&r^{1-n-2s}e^{-\frac{(n-1)r}{2}}
			\int_{0}^{r^2/4}t^{\frac n2+s-1}\exp\!\left\{-\,\frac{(n-1)^{2}r^2}{16t} \;-t\right\}
			\;\cdot\;(1+r+\frac{r^2}{4t})^{\frac{n-3}{2}} dt \\ \sim& r^{-s-1}e^{-\left(n-1\right)r}.
		\end{aligned}
		\]
		$I_2$ is the leading term for large $r$, since $I_1$ has additional Gaussian decay. Therefore
		\[
		\mathcal K_s(r)\sim r^{-1-s}\,e^{-(n-1)r},	\quad r\to\infty.
		\]
		Combining the two regimes completes the proof.
	\end{proof}

	We have thus derived the pointwise formula for the fractional Laplacian on $\mathbb{H}^n$ directly from the spectral functional calculus by means of the Bochner integral representation.  In particular, the above Proposition \ref{fensjx11} establishes the precise small– and large–distance asymptotics of the radial kernel $\mathcal{K}_s(r)$.

	\subsection{Pointwise representation of logarithmic Laplacian} 
	In this subsection we derive an explicit pointwise integral formula for the logarithmic Laplacian on \(\mathbb{H}^n\) and carry out a detailed asymptotic analysis of its kernel. The positive spectral bottom ensures convergence of the large-time logarithmic integral.
	
	The spectrum of the self-adjoint realization of $-\Delta_{\mathbb{H}^n}$ in $L^2(\mathbb H^n)$ is purely continuous:
	\[
	\sigma\bigl(-\Delta_{\mathbb{H}^n}\bigr)
	=\bigl[\rho^2,\infty\bigr),
	\qquad \rho=\tfrac{n-1}{2}.
	\]
	Define
	\[
	H^{\log}(\mathbb{H}^n)
	=\Bigl\{\,f\in L^2(\mathbb{H}^n)\;\Big|\;
	\int_{\rho^2}^{\infty}(\log\lambda)^2\,dE_{f,f}(\lambda)<\infty\Bigr\},
	\]
	with norm 
	\[
	\|f\|_{H^{\log}}^2
	=\|f\|_{L^2}^2
	+\int_{\rho^2}^{\infty}(\log\lambda)^2\,dE_{f,f}(\lambda).
	\]
	The spectral logarithmic Laplacian is
	\[
	\log(-\Delta_{\mathbb{H}^n})\,f
	:=\int_{\rho^2}^{\infty}\log\lambda\;dE(\lambda)\,f\in L^2\left(\mathbb{H}^n\right),
	\quad f\in H^{\log}(\mathbb{H}^n).
	\]
	Equivalently, by the Bochner integral formula one has in the \(L^2\) sense, see Theorem \ref{logboch11}
	\[
	\log(-\Delta_{\mathbb{H}^n})\,f
	=\int_{0}^{\infty}\frac{e^{-t}f - e^{t\Delta_{\mathbb{H}^n}}f}{t}\,dt.
	\]
	
	On $\mathbb H^n$, $\mathrm{Ric}_g=-(n-1)g$ and
$\inf\sigma(-\Delta)=(n-1)^2/4>0$. The spectral estimate following
Theorem~\ref{thm:log-noncompact11} verifies \eqref{heat-tail}.
Thus all hypotheses of that theorem are satisfied.  Hence we obtain:
	
	\begin{theorem}\label{point}
		Let $f\in C_c^{\alpha}\left(\mathbb{H}^n\right),\alpha>0$, then the logarithmic Laplacian admits a pointwise representation on \(\mathbb{H}^n\):
		\begingroup\small\[
		\begin{aligned}
			\bigl(\log(-\Delta_{\mathbb{H}^n})f\bigr)(x)=\int_{\mathbb{H}^n}K_1\bigl(r\bigr)\,\bigl(f(x)-f(y)\bigr)\,d\mathrm{vol}(y)
			-\int_{\mathbb{H}^n}K_2\bigl(r\bigr)\,f(y)\,d\mathrm{vol}(y)
			\;+\;\Gamma'(1)\,f(x),
		\end{aligned}
		\]
		\endgroup
		where
		\begin{equation}\label{loghe}
			K_1(r)
			=\int_{0}^{1}\frac{p_n(r,t)}{t}\,dt,
			\qquad
			K_2(r)
			=\int_{1}^{\infty}\frac{p_n(r,t)}{t}\,dt,
		\end{equation}
		and $p_n(r,t)$ is given by Theorem~\ref{heatfo}.
	\end{theorem}

	Below we present the asymptotic estimates for the logarithmic kernel functions.

	\begin{proof}[\textbf{Proof of
			Proposition~\ref{prop:K1K211}.}]
			
		From Proposition~\ref{prop:heat-asymp} we know that
		\begin{equation}\label{this est}
			p_{n}(r,t)\sim
			t^{-\,\frac{n}{2}}
			\exp\!\Bigl(-\tfrac{(n-1)^{2}}{4}\,t - \tfrac{r^{2}}{4t} - \tfrac{(n-1)r}{2}\Bigr)
			(1+r+t)^{\frac{n-3}{2}}(1+r),n\ge 2
		\end{equation}
		uniformly in $r\ge0,\;t>0$.
		
		We first estimate $K_1(r)$. Inserting \eqref{this est} with $t\le1$, we obtain
		\[
		K_{1}(r)
		\sim
		(1+r)^{\frac{n-1}{2}}
		\exp\!\Bigl(-\tfrac{(n-1)r}{2}\Bigr)\int_{0}^{1} t^{-\frac{n}{2}-1}
		\exp\!\Bigl(-\tfrac{r^{2}}{4t}\Bigr)\,dt.
		\]
		Set $s=\dfrac{r^{2}}{4t}$, then
		\[
		K_{1}(r)
		\sim
		(1+r)^{\frac{n-1}{2}}
		e^{-\frac{(n-1)r}{2}}\;
		r^{-n}
		\int_{r^{2}/4}^{\infty} s^{\frac{n}{2}-1}e^{-s}\,ds.
		\]
		If $r\le1$, the tail integral is a
		Gamma‐function bounded by a constant, giving
		$$K_{1}(r)\sim (1+r)^{\frac{n-1}{2}}r^{-n}
		\sim \,r^{-n}.$$  
		If $r\ge1$, by Lemma \ref{lem:gamma-tail} we have
		\[ \int_{r^{2}/4}^{\infty} s^{\frac{n}{2}-1}e^{-s}\,ds\sim r^{n-2}e^{-\frac{r^2}{4}}.\]
		Hence,
		\[
		K_{1}(r)
		\sim r^{\frac{n-5}{2}}
		e^{-\frac{(n-1)r}{2}}\;
		e^{-\frac{r^2}{4}}.
		\]
		
		Next, we estimate $K_2(r).$ Inserting \eqref{this est} with $t\ge1$, we obtain 
		\[
		K_{2}(r)
		\sim
		\left(1+r\right)
		\exp\!\Bigl(-\tfrac{(n-1)r}{2}\Bigr)\int_{1}^{\infty} t^{-\frac{n}{2}-1}
		\exp\!\left\{-\,\frac{(n-1)^{2}}{4}\,t \;-\;\frac{r^{2}}{4t}\right\}\left(1+r+t\right)^{\frac{n-3}{2}}\,dt.
		\]
		Set $s=\dfrac{r^{2}}{4t}$, then
		\[
		K_{2}(r)
		\sim
		\left(1+r\right)r^{-n}e^{-\frac{(n-1)r}{2}}
		\int_{0}^{r^2/4}t^{\frac n2-1}\exp\!\left\{-\,\frac{(n-1)^{2}r^2}{16t} \;-t\right\}
		\;\cdot\;(1+r+\frac{r^2}{4t})^{\frac{n-3}{2}} dt.
		\]
		By Lemma \ref{lem:K2-asymp}, if $r\ge 1,$ $K_2(r)\sim r^{-1}e^{-\left(n-1\right)r},$
		and $K_2(r)\sim 1, \:r\le 1.$
	\end{proof}
	
	\begin{remark}
		On the Euclidean space \(\mathbb{R}^n\), the heat kernel has the explicit Gaussian form
		\[
		p_n(r,t)\;=\;(4\pi t)^{-n/2}\,e^{-\,\frac{r^2}{4t}},
		\qquad r>0,\;t>0.
		\]
		Substituting this into \eqref{loghe}, we obtain
		\[
		K_1(r)
		=\int_0^1 \frac{p_n(r,t)}{t}\,dt
		=(4\pi )^{-n/2}\int_0^1 \,t^{-1-\frac{n}{2}}\,e^{-r^2/(4t)}\,dt=\pi^{-\frac{n}{2}}r^{-n}\int_{\frac{r^2}{4}}^{\infty}t^{\frac{n}{2}-1}e^{-t}dt,\]
		and
		\[
		K_2(r)
		=\int_1^\infty \frac{p_n(r,t)}{t}\,dt
		=(4\pi )^{-n/2}\int_1^\infty t^{-1-\frac{n}{2}}\,e^{-r^2/(4t)}\,dt=\pi^{-\frac{n}{2}}r^{-n}\int_{0}^{\frac{r^2}{4}}t^{\frac{n}{2}-1}e^{-t}dt.
		\]
		For large \(r\ge1\), the integrals exhibit different asymptotic behaviors. By Lemma \ref{lem:gamma-tail},
		
		\begin{itemize}
			\item[(i)] \textbf{Small-scale part \(K_1(r)\):} exponential decay dominates: 
			\[
			K_1(r)
			\;\sim\;
			r^{-2}e^{-\frac{r^2}{4}}.
			\]
			\item[(ii)] \textbf{Large-scale part \(K_2(r)\):} $K_2(r)
			\;\sim\;
			r^{-n}.$
		\end{itemize}
        \smallskip
        
		For small \(r\le1\),
		
		\begin{itemize}
			\item[(iii)] \textbf{Small-scale part \(K_1(r)\):} $K_1(r)
			\;\sim\;
			r^{-n}.$
            \smallskip
			\item[(iv)] \textbf{Large-scale part \(K_2(r)\):} 
			\[
			K_2(r)
			\;\sim\;
			r^{-n}\int_0^{\frac{r^2}{4}}t^{\frac{n}{2}-1}dt\sim 1.
			\]
		\end{itemize}
	\end{remark}
	
	\subsection{Estimates for kernels and regularity of logarithmic Laplacian}\label{last}
	
	For a continuous (hence locally uniformly bounded) function \(f:\Omega\to \mathbb{R}\), fix $x\in \Omega$ and define the pointwise modulus of continuity on $\Omega$
	\[
	\omega_{f,x,\Omega}(r)
	\;:=\;\sup_{\substack{y\in\Omega\\d(x,y)\le r}}
	\bigl|f(y)-f(x)\bigr|,
	\quad r>0.
	\]
	We say \(f\) is Dini continuous at the point \(x\) on $\Omega$ if
	\[
	\int_{0}^{1}\frac{\omega_{f,x,\Omega}(r)}{r}\,dr<\infty.
	\]
	We say that $f$ is locally Dini continuous at a point $x$ on $\Omega$ if there exists an open neighborhood $x\in \Omega_0\subset \Omega$ such that
	$$
	\int_{0}^{1}\frac{\omega_{f,x,\Omega_0}(r)}{r}\,dr<\infty.
	$$
	
	Define the global modulus over \(\Omega\) by
	\[
	\omega_{f,\Omega}(r)
	\;:=\;\sup_{x\in\Omega}\omega_{f,x,\Omega}(r).
	\]
	We say \(f\) is uniformly Dini continuous on \(\Omega\) if
	\[
	\int_{0}^{1}\frac{\omega_{f,\Omega}(r)}{r}\,dr<\infty.
	\]
	
	\begin{proposition}\label{dinia}
		For all $x\in\mathbb H^{n}$, the integral
		\[
		\int_{\mathbb H^{n}}K_{1}\!\bigl(d_{\mathbb H^{n}}(x,y)\bigr)\,\bigl(f(x)-f(y)\bigr)\,d\mathrm{vol}(y)
		\]
		is finite when $f$ is bounded in $\mathbb H^{n}$
		and \(f\) is locally Dini continuous at the point \(x\) on $\mathbb H^{n}$.
	\end{proposition}
	
	\begin{proof}
		By Proposition \ref{prop:K1K211}, the kernel satisfies
		\[
		K_{1}(r)\;\sim\;
		\begin{cases}
			r^{-n}, & 0<r\le1,\\[4pt]
			r^{\frac{n-5}{2}}\exp\!\bigl(-\tfrac{(n-1)r}{2}\bigr)\exp\!\bigl(-\tfrac{r^{2}}{4}\bigr), & r\ge1.
		\end{cases}
		\]
		
		Let $x\in\mathbb H^{n}$ be fixed and set $r=d_{\mathbb H^{n}}(x,y)$.  
		Write the polar decomposition of volume as \(d\mathrm{vol}_{\mathbb H^{n}}=\sinh^{n-1}(r)\,dr\,d\omega_{n-1}\) , where
		\[
		\sinh^{n-1}(r)\;\sim\;
		\begin{cases}
			r^{\,n-1}, & 0<r\le1,\\[4pt]
			e^{(n-1)r}, & r\ge1.
		\end{cases}
		\]
		Since \(f\) is locally Dini continuous at the point \(x\) on $\mathbb H^{n}$, then there exists $\Omega\subset \mathbb{H}^n,$ such that
		\[
		\int_{0}^{1}\frac{\omega_{f,x,\Omega}(r)}{r}\,dr<\infty.
		\]
		\noindent\textbf{Near the diagonal: $0<r\le1$.}  
		Since $K_{1}(r)\sinh^{n-1}(r)\sim r^{-1}$, for sufficiently small $\delta>0,$ we have
		\[
		\int_{B_\delta(x)}K_{1}(r)\,|f(x)-f(y)|\,d\mathrm{vol}(y)
		\;\lesssim \int_{0}^{\delta}\frac{\omega_{f,x,\Omega}(r)}{r}\,dr<\infty.
		\]
		\noindent\textbf{Far from the diagonal: $r\ge1$.}  
		Now
		$$K_{1}(r)\sinh^{n-1}(r)
		\sim 	r^{\frac{n-5}{2}}\exp\!\bigl(\tfrac{(n-1)r}{2}\bigr)\exp\!\bigl(-\tfrac{r^{2}}{4}\bigr)$$
		whose Gaussian decay ensures for every $f\in L^{\infty}\left(\mathbb{H}^n\right)$, we have
		\begin{equation}\label{k1deguji}
			\int_{\mathbb H^n\setminus B_\delta(x)}K_1(r)|f(x)-f(y)|\,d\mathrm{vol}(y)<\infty.
		\qedhere
		\end{equation}
	\end{proof}
	
	\begin{proposition}[Integrability of $K_2$ on $\mathbb H^n$]\label{k2kjx}
Let $K_2$ be defined in \eqref{loghe}, and let $1\le p\le\infty$. Then		\[
		K_{2}\in L^{p}(\mathbb H^{n})\quad\Longleftrightarrow\quad p>1.
		\]
	\end{proposition}
	\begin{proof}
		By Proposition \ref{prop:K1K211},
		\[
		K_{2}(r)\;\sim\;
		\begin{cases}
			1, & 0<r\le1 \\
			r^{-1}\exp\!\bigl(-(n-1)r\bigr), & r\ge1.
		\end{cases}
		\]
		Using the polar measure \(d\mathrm{vol}_{\mathbb H^{n}}=\sinh^{n-1}(r)\,dr\,d\omega_{n-1}\) and
		\(\sinh^{n-1}(r)\sim r^{\,n-1}\) as \(r\to0\), \(\sinh^{n-1}(r)\sim e^{(n-1)r}\) as \(r\to\infty\), for $1\le p<\infty$ we have
		\[
		\|K_2\|_{L^p(\mathbb{H}^n)}^p
		\;\sim\;
		\int_1^\infty r^{-p}e^{(n-1)r(1-p)}\,dr+1.
		\]
		For $1\le p<\infty$, this integral is finite exactly when $p>1$. The kernel is also bounded, so $K_2\in L^\infty$. Thus, $	K_{2}\in L^{p}(\mathbb H^{n})\Longleftrightarrow p>1.$
	\end{proof}

	\begin{remark}\label{konhgj}
		Since \(K_{2}\notin L^{1}(\mathbb H^{n})\), then
		\[
		\int_{\mathbb H^{n}}K_{2}\bigl(d(x,y)\bigr)\,f(y)\,d\mathrm{vol}(y)
		\]
		diverges in general for merely bounded smooth \(f\).  To make it finite one needs the weighted integrability
		\[
		f\in L^{1}_{w,x_{0}}(\mathbb H^{n})
		:=\Bigl\{\,f:\mathbb H^{n}\to\mathbb{R}:\ 
		\int_{\mathbb H^{n}}
		\frac{|f(y)|}{\bigl(1+d(x_{0},y)\bigr)}
		\,e^{-\left(n-1\right)d(x_{0},y)}\,
		d\mathrm{vol}(y)
		<\infty
		\Bigr\}.
		\]
	\end{remark}

	\begin{remark}
		Changing the base point from $x_{0}$ to any other $x_{1}\in\mathbb H^n$ yields an equivalent norm, so the space does not depend on the choice of $x_0$.  Hence we may simply write
		\[
		C_c\left(\mathbb{H}^n\right)\subset L^{1}(\mathbb H^{n}) \subset L^{1}_{w}(\mathbb H^{n})
		:=L^{1}_{w,x_{0}}(\mathbb H^{n})
		\]
		with the understanding that different base points give the same normed space.
	\end{remark}
	
	By carefully examining the proof of Proposition~\ref{dinia} (in particular step~(\ref{k1deguji})) and the definition of the weighted space \(L^1_w(\mathbb{H}^n)\), we immediately obtain:
	
	\begin{corollary}
		Proposition \ref{dinia} holds when $f\in L_{w}^1(\mathbb H^n)$ and \(f\) is locally Dini continuous at the point \(x\) on $\mathbb H^{n}$.
	\end{corollary}
	
	The formula in Theorem \ref{point} allows us to define $\log\left(-\Delta_{\mathbb{H}^n}\right)f$ for a fairly large class of functions, see Corollary \ref{larclas}.
	
	\begin{corollary}\label{larclas}
		If $f\in L_{w}^1(\mathbb H^n)$ and \(f\) is locally Dini continuous at the point \(x\) on $\mathbb H^{n}$, then $\left(\log\left(-\Delta_{\mathbb{H}^n}\right)f\right)(x)$ is well defined in Theorem \ref{point}.
	\end{corollary}
	
	Now, we give the proof of Proposition \ref{prop:cont-logDelta11}.
		
\begin{proof}[\textbf{Proof of
		Proposition~\ref{prop:cont-logDelta11}.}]
			Absolute convergence at every Dini point follows from
			Proposition~\ref{dinia}, its weighted corollary, and the estimates for
			$K_2$. Assume that $f$ is locally uniformly Dini continuous. Fix
			$x\in\mathbb H^n$ and let $x_j\to x$. Choose a relatively compact
			neighborhood $\Omega$ of $x$ such that
			\[
			\int_0^1\frac{\omega_{f,\Omega}(r)}{r}\,dr<\infty.
			\]
			For all sufficiently large $j$, $x_j\in\Omega$. If $\delta>0$ is small
			enough, then $B_\delta(x_j)\subset\Omega$, and
			Proposition~\ref{dinia} gives
			\[
			\int_{B_\delta(x_j)}
			K_1\bigl(d(x_j,y)\bigr)
			|f(x_j)-f(y)|\,d\mathrm{vol}(y)
			\le
			C\int_0^\delta\frac{\omega_{f,\Omega}(r)}{r}\,dr.
			\]
			Hence the contribution of $B_\delta(x_j)$ tends to zero as
			$\delta\downarrow0$, uniformly in $j$.
			
			For the complementary part, set
			\[
			F_j(y):=
			\mathbf1_{\{d(x_j,y)\ge\delta\}}
			K_1\bigl(d(x_j,y)\bigr)
			\bigl[f(x_j)-f(y)\bigr].
			\]
			For almost every $y\in\mathbb H^n$,
			\[
			F_j(y)\longrightarrow
			\mathbf1_{\{d(x,y)\ge\delta\}}
			K_1\bigl(d(x,y)\bigr)
			\bigl[f(x)-f(y)\bigr],
			\]
			since the sphere $\{y:d(x,y)=\delta\}$ has zero volume.
			
			Let $K\Subset\mathbb H^n$ contain $x$ and all sufficiently large
			$x_j$. By Proposition~\ref{prop:K1K211} and
			Remark~\ref{konhgj}, the kernel $K_1(d(z,y))$ is uniformly bounded
			for $z\in K$ on compact sets separated from the diagonal, while at
			infinity
			\[
			K_1\bigl(d(z,y)\bigr)
			\le
			G(y)
			+
			C_K\frac{e^{-(n-1)d(x_0,y)}}{1+d(x_0,y)},
			\qquad z\in K,
			\]
			where $G\in L^1(\mathbb H^n)$. Since $f$ is continuous on $K$ and
			$f\in L^1_w(\mathbb H^n)$, this yields an integrable majorant for
			$|F_j|$, independent of $j$. Hence, by dominated convergence,
			\[
			\int_{\mathbb H^n\setminus B_\delta(x_j)}
			K_1\bigl(d(x_j,y)\bigr)
			\bigl[f(x_j)-f(y)\bigr]\,d\mathrm{vol}(y)
			\]
			converges to the corresponding integral with $x_j$ replaced by $x$.
			Letting $\delta\downarrow0$ and using the preceding uniform estimate,
			we obtain continuity of
			\[
			x\longmapsto
			\int_{\mathbb H^n}
			K_1\bigl(d(x,y)\bigr)
			\bigl[f(x)-f(y)\bigr]\,d\mathrm{vol}(y).
			\]
			
			Next consider
			\[
			T_2f(x):=
			\int_{\mathbb H^n}
			K_2\bigl(d(x,y)\bigr)f(y)\,d\mathrm{vol}(y).
			\]
			For every compact $K\Subset\mathbb H^n$, Proposition~\ref{prop:K1K211}
			and the triangle inequality give
			\[
			K_2\bigl(d(x,y)\bigr)
			\le
			C_K\frac{e^{-(n-1)d(x_0,y)}}{1+d(x_0,y)},
			\qquad x\in K.
			\]
			Moreover, $K_2$ is continuous on $[0,\infty)$. Thus, if $x_j\to x$ in
			$K$,
			\[
			K_2\bigl(d(x_j,y)\bigr)f(y)
			\longrightarrow
			K_2\bigl(d(x,y)\bigr)f(y)
			\]
			for every $y$, and the preceding estimate gives an $L^1$ majorant.
			Dominated convergence therefore yields $	T_2f(x_j)\longrightarrow T_2f(x).$
			In addition,
			\[
			\|T_2f\|_{L^\infty(K)}
			\le
			C_K\|f\|_{L^1_w(\mathbb H^n)}.
			\]
			
			Finally, the term $\Gamma'(1)f$ is continuous. Hence the pointwise
			representation of $\log(-\Delta_{\mathbb H^n})f$ is continuous on
			$\mathbb H^n$.
		\end{proof}

	Fix \(x\in\mathbb H^{n}\), let  
	\[
	B_{1}(x)=\{y\in\mathbb H^{n}:d(x,y)<1\}, 
	\qquad r=d(x,y),
	\]
	define the remainder
	\begingroup\small	
	\begin{equation}\label{remind}
		\begin{aligned}
			\mathcal R_{n}(f;x)
			:=&-\int_{B_{1}(x)}K_{2}(r)\,f(y)\,d\mathrm{vol}(y)
			-\int_{\mathbb H^{n}\setminus B_{1}(x)}
			K_{1}(r)f(y)\,d\mathrm{vol}(y)\\+&\left(\int_{\mathbb H^{n}\setminus B_{1}(x)}	K_{1}(r)\,d\mathrm{vol}(y)+\Gamma'(1)\right)f(x).
		\end{aligned}
	\end{equation}
	\endgroup
	
	Write $\chi_1(r)=\mathbf1_{\{r<1\}}$.
By Proposition~\ref{prop:K1K211}, both radial kernels
$\chi_1K_2$ and $\mathbf1_{\{r\ge1\}}K_1$ belong to $L^1(\mathbb H^n)$.
For either kernel $k$, homogeneity gives the same finite row and column
integral $M_k=\int_{\mathbb H^n}|k(d(o,y))|\,d\mathrm{vol}(y)$.
Fubini gives the $L^1$ bound and direct integration the $L^\infty$ bound;
interpolation gives
\[
\|T_kf\|_p\le M_k\|f\|_p,\qquad1\le p\le\infty,
\quad T_kf(x)=\int_{\mathbb H^n}k(d(x,y))f(y)\,d\mathrm{vol}(y).
\]
Consequently $\mathcal R_n$ is bounded on every such $L^p$. 
	
	We therefore obtain the following splitting on $\mathbb H^n$, see Proposition \ref{prop:split-refined11}, whose three terms correspond exactly to the three parts in the Euclidean formula, 
	\begin{equation}\label{lohrn}
		\log \left(-\Delta\right)f(x)  =c_{N}\int_{B_{1}(x)}\frac{f(x)-f(y)}{|x-y|^{N}}\,dy
		-c_{N}\int_{\mathbb{R}^{N}\setminus B_{1}(x)}
		\frac{f(y)}{|x-y|^{N}}\,dy
		+\rho_{N}\,f(x), \quad x\in\mathbb{R}^{N}, 
	\end{equation}
	where
	\[
	c_{N}:=\pi^{-N/2}\Gamma\!\bigl(\tfrac{N}{2}\bigr)
	=\frac{2}{|\mathbb{S}^{N-1}|}, 
	\qquad
	\rho_{N}:=2\log 2+\psi\!\bigl(\tfrac{N}{2}\bigr)-\gamma,           
	\]
	\(\gamma=-\Gamma'(1)\) is the Euler–Mascheroni constant, and
	\(\psi=\Gamma'/\Gamma\) is the digamma function, which is given by H.~Chen and T.~Weth in \cite[Theorem 1.1]{chen2019dirichlet}.
	
	Based on the above analysis, we present the following proof.

\begin{proof}[\textbf{Proof of
		Proposition~\ref{prop:split-refined11}.}]
		The proof is straightforward: one uses the local Dini condition on $f$ at $x$, the argument of Proposition \ref{dinia}, the kernel estimates from Proposition \ref{prop:K1K211}, and the definitions of the exponential decay space $L_{w}^1(\mathbb H^n)$ to check that each of the integrals is finite and hence the representation is well defined.
	\end{proof}

	\begin{proposition}\label{prop:R-L2}
		Let $n\ge2$ and $f\in L^p(\mathbb H^n)\cap L^2(\mathbb H^n)$. For the remainder \(\mathcal R_{n}(f;x)\) defined by \eqref{remind} one has 
		\[
		\int_{\mathbb H^{n}}\bigl|\mathcal R_{n}(f;x)\bigr|\;|f(x)|\;
		d\mathrm{vol}(x)\le
		\|\chi_{1}K_2\|_{q}\
		\|f\|_{p}^{2}+\bigl\|K_{1}\mathbf{1}_{\{r\ge1\}}\bigr\|_{q}\,\|f\|_{p}^{2}+|\rho_n|\,\|f\|_{2}^{2},
		\]
		where $K_1,K_2$ are defined in Proposition \ref{prop:K1K211},
		\[\rho_n:=\int_{\mathbb H^{n}\setminus B_{1}(x)}	K_{1}(r)\,d\mathrm{vol}(y)+\Gamma'(1)<\infty,\]
		and
		\[\frac{1}{q}+\frac{2}{p}=2,\:\:p,q\ge1.\]
	\end{proposition}
	
	\begin{proof}
		Decompose \(\mathcal R_{n}(f;x)=R_{1}(x)+R_{2}(x)+R_{3}(x)\) with
		\begingroup\small	\[
		R_{1}(x)=-\int_{B_{1}(x)}K_{2}(r)\,f(y)\,d\mathrm{vol}(y),\quad
		R_{2}(x)=-\int_{\mathbb H^{n}\setminus B_{1}(x)}K_{1}(r)\,f(y)\,d\mathrm{vol}(y),\quad
		R_{3}(x)=\rho_{n}\,f(x),
		\]
		where 
		\[\rho_n:=\int_{\mathbb H^{n}\setminus B_{1}(x)}	K_{1}(r)\,d\mathrm{vol}(y)+\Gamma'(1)<\infty.\]
		\endgroup
		
		For a radial kernel $k\in L^q(\mathbb H^n)$, $1\le q\le\infty$,
homogeneity and Minkowski give
$\|T_k\|_{L^1\to L^q}\le\|k\|_q$, while H\"older gives
$\|T_k\|_{L^{q'}\to L^\infty}\le\|k\|_q$.
Interpolation halfway between these bounds yields
\[
\|T_kf\|_{p'}\le\|k\|_q\|f\|_p,
\qquad \frac1q+\frac2p=2.
\]
For $q=\infty$, this is just the direct $L^1\to L^\infty$ bound.
Pairing with $|f|$ and taking $k=\chi_1K_2$ and
$k=\mathbf1_{\{r\ge1\}}K_1$ bounds the first two terms by
their respective $L^q$ norms times $\|f\|_p^2$.
The third contribution is exactly $|\rho_n|\|f\|_2^2$.
The triangle inequality proves the proposition.
	\end{proof}

	Finally, we analyze the \(L^p\)-integrability of the logarithmic Laplacian.

	\begin{proof}[\textbf{Proof of
			Proposition~\ref{logint11}.}]
Let $S=\operatorname{supp}f$ and
$S_1=\{x:d(x,S)\le1\}$. This set is compact. For
\[
T_1f(x)=\int_{B_1(x)}K_1(d(x,y))[f(x)-f(y)]\,d\mathrm{vol}(y),
\]
the uniform Dini condition gives
\[
\|T_1f\|_\infty\le C\int_0^1\frac{\omega_{f,\mathbb H^n}(r)}r\,dr<\infty.
\]
Also $T_1f=0$ outside $S_1$, so $T_1f\in L^p$ for $1\le p\le\infty$.
Put $j(r)=\mathbf1_{\{r\ge1\}}K_2(r)$. For $1<p<\infty$,
Minkowski and homogeneity give
\[
\left\|\int_{\mathbb H^n}j(d(\cdot,y))f(y)\,d\mathrm{vol}(y)\right\|_p
\le\int_{\mathbb H^n}\|j(d(\cdot,y))\|_p|f(y)|\,d\mathrm{vol}(y)
=\|j\|_p\|f\|_1<\infty.
\]
For $p=\infty$, the bound is $\|j\|_\infty\|f\|_1$.
The remainder belongs to every $L^p$, $1\le p\le\infty$, by the
row and column integral estimates above and $f\in C_c$.
Apply Proposition~\ref{prop:split-refined11} to conclude.
\end{proof}

	At last, we introduce on \(\mathbb H^n\) two nonnegative kernels
	\[
	k_{\mathbb H^n}(x,y)
	:=K_{1}\bigl(d(x,y)\bigr)\,\mathbf1_{B_{1}(x)}(y),
	\qquad
	j_{\mathbb H^n}(x,y)
	:=K_{2}\bigl(d(x,y)\bigr)\,\mathbf1_{\mathbb H^n\setminus B_{1}(x)}(y).
	\]
	Then the pointwise representation \eqref{eq:split-refined} can be written as
	\begin{equation}\label{leisiyur}
		\log(-\Delta_{\mathbb H^n})f(x)
		=\int_{\mathbb H^n}\bigl[f(x)-f(y)\bigr]\,k_{\mathbb H^n}(x,y)\,d\mathrm{vol}(y)
		\;-\;\bigl[j_{\mathbb H}*f\bigr](x)
		\;+\;\mathcal R_n(f;x),
	\end{equation}
	where
	\[
	\bigl[j_{\mathbb H}*f\bigr](x)
	:=\int_{\mathbb H^n}j_{\mathbb H}(x,y)\,f(y)\,d\mathrm{vol}(y),
	\]
	and the remainder \(\mathcal R_n(f;x)\) is defined in \eqref{remind}.  As in the Euclidean case, \(k_{\mathbb H^n}\) encodes the singular “nonlocal” part on the unit ball, whereas \(j_{\mathbb H}\) captures the “nonlocal” tail outside \(B_1(x)\). By Remark \ref{konhgj}, the convolution–type operators are well defined under hypotheses on \(f\) in  Proposition \ref{prop:split-refined11}.
	
	This is analogous to the representation of the logarithmic Laplacian on \(\mathbb{R}^n\), and in this form we are well–positioned to study the associated variational problems. Recall that
	\[	k:\mathbb{R}^n\setminus\{0\}\to\mathbb{R},\qquad
	k(z)=c_n\,\mathbf1_{B_1}(z)\,\lvert z\rvert^{-n},\]
	\[	j:\mathbb{R}^n\to\mathbb{R},\qquad
	j(z)=c_n\,\mathbf1_{\mathbb{R}^n\setminus B_1}(z)\,\lvert z\rvert^{-n},\]
	where \(c_n=\pi^{-n/2}\Gamma(n/2)\).  Then the integral representation (\ref{lohrn}) can be rewritten as
	\begin{equation}\label{eq:euclid-loglap}
		\log \left(-\Delta\right)f(x)
		=\int_{\mathbb{R}^n}\bigl(f(x)-f(y)\bigr)\,k(x-y)\,dy
		\;-\;\bigl[j * f\bigr](x)
		\;+\;\rho_n\,f(x),
	\end{equation}
	with
	\[
	\bigl[j * f\bigr](x)
	=\int_{\mathbb{R}^n}j(x-y)\,f(y)\,dy,
	\quad
	\rho_ n
	=2\log2 + \psi\!\bigl(\tfrac n2\bigr)-\gamma,
	\]
	where \(\psi=\Gamma'/\Gamma\) is the digamma function and \(\gamma\) the Euler–Mascheroni constant.

	\section*{Acknowledgements}
	
	We thank Zhiqin Lu for providing us with the key idea of defining the logarithmic Laplacian via functional calculus on general manifolds. We are deeply grateful to Bobo Hua for his invaluable guidance and insightful discussions, and to Tobias Weth and Huyuan Chen for their in-depth conversations on the properties of the logarithmic Laplacian. We also thank Zuoqin Wang  for his careful reading and constructive feedback on the paper, and Wancheng Zhang for his valuable insights into heat kernel estimates on Riemannian manifolds.
	
%\bibliographystyle{amsplain}
%\bibliography{refs}
%\enlargethispage{2\baselineskip}
	
	\providecommand{\bysame}{\leavevmode\hbox to3em{\hrulefill}\thinspace}
	\providecommand{\MR}{\relax\ifhmode\unskip\space\fi MR }
	% \MRhref is called by the amsart/book/proc definition of \MR.
	\providecommand{\MRhref}[2]{%
		\href{http://www.ams.org/mathscinet-getitem?mr=#1}{#2}
	}
	\providecommand{\href}[2]{#2}

\end{document}